\documentclass[letterpaper]{article}
\usepackage{anysize} 
\usepackage{amsfonts, amsmath, amsthm, amssymb, amscd} 
\usepackage{enumerate}
\usepackage{comment}
\usepackage{subfig}
\usepackage[english]{babel}
\usepackage[utf8]{inputenc}
\usepackage[T1]{fontenc}
\usepackage{helvet}
\usepackage{etoolbox}
\usepackage{graphicx}
\usepackage{titlesec}
\usepackage{caption}
\usepackage{booktabs}
\usepackage{xcolor} 
\usepackage{caption}
\graphicspath{{./source_code/figures/}}
\usepackage{scrextend}
\usepackage{fancyhdr}
\usepackage{graphicx}
\usepackage[nottoc]{tocbibind}

\newtheorem{definition}{Definition}[section]

\newtheorem{lemma}[definition]{Lemma}
\newtheorem{proposition}[definition]{Proposition}

\newtheorem{remark}[definition]{Remark}

\fancypagestyle{plain}{
	\fancyhf{}

}
\makeatletter
\patchcmd{\@maketitle}{\LARGE \@title}{\fontsize{16}{19.2}\selectfont\@title}{}{}
\makeatother

\usepackage{authblk}

\newsavebox\affbox
\author[ ]{\textbf{Renato Calleja}\thanks{calleja@mym.iimas.unam.mx}, \textbf{Pablo Padilla-Longoria}\thanks{pablo@mym.iimas.unam.mx} and \textbf{Edgar Rodr\'iguez-Mendieta}\thanks{Corresponding author, edgarm0509@ciencias.unam.mx}}

\affil[$\ast$]{Depto. Matem\'aticas y Mec\'anica, IIMAS, Universidad Nacional Aut\'onoma de M\'exico, Apdo. Postal 20-126, C.P. 04510, M\'exico CDMX, M\'exico.}

\affil[$\dagger$]{Depto. Matem\'aticas y Mec\'anica, IIMAS, Universidad Nacional Aut\'onoma de M\'exico, Apdo. Postal 20-126, C.P. 04510, M\'exico CDMX, M\'exico.}

\affil[$\ddagger$]{Depto. Matem\'aticas y Mec\'anica, IIMAS, Universidad Nacional Aut\'onoma de M\'exico, Apdo. Postal 20-126, C.P. 04510, M\'exico CDMX, M\'exico.}

\titlespacing\section{0pt}{12pt plus 4pt minus 2pt}{0pt plus 2pt minus 2pt}
\titlespacing\subsection{12pt}{12pt plus 4pt minus 2pt}{0pt plus 2pt minus 2pt}
\titlespacing\subsubsection{12pt}{12pt plus 4pt minus 2pt}{0pt plus 2pt minus 2pt}

\titleformat{\section}{\normalfont\fontsize{10}{15}\bfseries}{\thesection.}{1em}{}
\titleformat{\subsection}{\normalfont\fontsize{10}{15}\bfseries}{\thesubsection.}{1em}{}
\titleformat{\subsubsection}{\normalfont\fontsize{10}{15}\bfseries}{\thesubsubsection.}{1em}{}

\titleformat{\author}{\normalfont\fontsize{10}{15}\bfseries}{\thesection}{1em}{}

\title{
\textbf{
\huge{ 
Biological network dynamics: Poincar\'{e}-Lindstedt series and the effect of delays
}
}
}
\date{\today}  
\begin{document}
\pagestyle{headings}	
\newpage
\setcounter{page}{1}
\renewcommand{\thepage}{\arabic{page}}
\captionsetup[figure]{labelfont={bf},labelformat={default},labelsep=period,name={Figure}}\captionsetup[table]{labelfont={bf},labelformat={default},labelsep=period,name={Table}}
\setlength{\parskip}{0.5em}
\maketitle
\begin{center}
\noindent\rule{15cm}{0.5pt}
\end{center}
\begin{abstract}
This paper focuses on the Hopf bifurcation in an activator-inhibitor system without diffusion which can be modeled as a delay differential equation. The main result of this paper is the existence of the Poincar\'e-Lindstedt series to all orders for the bifurcating periodic solutions. The model has a non-linearity which is non-polynomial, and yet this allows us to exploit the use of Fourier-Taylor series to develop order-by-order calculations that lead to linear recurrence equations for the coefficients of the Poincar\'e-Lindstedt series. As applications, we implement the computation of the coefficients of these series for any finite order, and use a pseudo-arclength continuation to compute branches of periodic solutions.


%

\textbf{\textit{Keywords}}: \textit{Delay differential equation, Hopf bifurcation, Poincar\'e-Lindstedt series.}
\end{abstract}
	
\begin{center}
	\noindent\rule{15cm}{0.4pt}
\end{center}

\section{Introduction}
Pattern formation in living systems is one of the central problems in developmental biology (see \cite{murrayII}). The mechanisms underlying the decoding of genetic information and how this information determines the emergence of structures (phenotype) is still widely unknown. In recent years, many genetic regulatory networks, GRNs, have been proposed as a model to understand this genotype to phenotype maping \cite{gierer_meinhardt1972}. On the other hand, reaction-diffusion systems have been advanced as models for biological pattern formation since the work by Turing \cite{turing52}. In fact, in his work Turing suggests that the diffusive chemicals, morphogenes, might be associated with genes. It is natural to associate the expression level of a gene with the concentration of the corresponding protein. This observation automatically links a GRN with a reaction diffusion system. In particular, in the present paper, we study one of the most important GRNs, namely, the activator-inhibitor system. There are already works devoted to this subject. However, the role played by delays deserves attention. Indeed, due to many factors, including the fact that the morphogenes take a positive time to diffuse, delays occur naturally in genetic networks. As it is well known in control engineering, delays can completely change the dynamical and stability properties of a system. The mathematical model that we will study is based in the Gierer-Meinhardt activator-inhibitor system studied in \cite{murrayII}, that is a reaction-diffusion system without time-delay. For the purposes of this paper, we consider a Gierer-Meinhardt model without spatial variation and two discrete delays,
\begin{equation*}
	\left.
	\begin{array}{rcl}
		P'(t) 
		& = 
		& \dfrac{k_{3} P^{2}(t)}{Q(t-\tau_{P})} - k_{2}P(t) + k_{1},	\\ \\    
		Q'(t) 
		& = 
		& k_{4} P^{2}(t-\tau_{Q}) - k_{5} Q(t) + k_{6},
	\end{array}
	\right.
\end{equation*}
where $P$ and $Q$ are two chemical species and $k_{1}, k_{2},k_{3},k_{4},k_{5}> 0 $. Discrete delays $\tau_{P},\tau_{Q} > 0$ are related with the fact of the interaction between $P$ and $Q$ is delayed by a time (a case of equal delays is studied in \cite{lee_gaffney_monk}). The constant $k_{6}\geq 0$ is related with the initial presence of any of the two chemical species (see \cite{gierer_meinhardt1972}).

Biological arguments allow us to suppose that $\tau_{P} = s_{0}\tau_{Q}$ with $s_{0}>1$. By a rescaling the $t$ variable, the above equation is equivalent to,
\begin{equation}\label{DDE_GM_noncentered_delays_1_s0}
	\left.
	\begin{array}{rcl}
		P'(t)
		& = 
		& \gamma \left[ \dfrac{P^{2}(t)}{Q\left(t-s_{0} \right)} - bP(t) + a \right], \\ \\
		Q'(t) 
		& = 
		& \gamma \left[P^{2}\left(t-1 \right) - Q(t) + c \right],
\end{array}
\right.
\end{equation}
where $a, b, \gamma > 0$ and $c\geq 0$ (note that $P$ and $Q$ have been rescaled but we kept the same notation).

An equilibrium $\left(u_{0}, v_{0}\right)^{T} \in \mathbb{R}^{2}$ of (\ref{DDE_GM_noncentered_delays_1_s0}), with $u_0, v_{0}\neq 0$,  satisfy,
\begin{subequations}\label{eq:equilibrium_u0_v0_of_DDE_noncentred_delays_1_s0} 
	\begin{align} 
		u_{0}^{3}  + c u_{0}
		&=
		\dfrac{a+1}{b} u_{0}^{2} + \dfrac{a c}{b}, \label{eq:equilibrium_u0}\\
		v_{0}
		&=
		u_{0}^{2} + c, \label{eq:equilibrium_v0}
	\end{align} 
\end{subequations}
from which we can find pairs $\left(u_{0}, v_{0} \right)$ in the first quadrant. Thus, considering $u(t) = P(t) - u_{0}$ and $v(t) = Q(t) - v_{0}$, system (\ref{DDE_GM_noncentered_delays_1_s0}) is equivalent to the following delay differential equation around $0$,
\begin{equation}\label{DDE_GM_around_0}
	\left.
	\begin{array}{rcl}
     u'(t)
     & = 
     & \gamma \left\lbrace \dfrac{\left[ u(t) + u_{0} \right]^{2}}{v\left(t-s_{0} \right) + v_{0}} - b\left[u(t)+u_{0}\right] + a \right\rbrace, \\ \\     
     
     v'(t) 
     & = 
     & \gamma \left\lbrace \left[ u\left(t-1 \right) + u_{0}\right]^{2} - \left[v(t) + v_{0}\right] + c \right\rbrace.
	\end{array}
\right.
\end{equation}

It is known that finding an explicit periodic solution of a delay differential equation is very complicated. However, there are methods (analytical and numerical) that allow us approximating periodic solutions, in our case using the Poincar\'e-Lindstedt series and the collocation method. The Poincar\'e-Lindstedt series is a perturbative method that, using Fourier-Taylor series, allow us analytically approximating a periodic solution by solving, order-by-order, a delay differential equation. On the other hand, the collocation method is a numerical method that allow us aproximanting points on the curve with sufficient accuracy. In this work both methods complement each other, since we use the Poincar\'e-Lindstedt series (at order $\varepsilon^{3}$ are sufficient in our simulations) as the initial guess of the Newton method generated by the collocation method.

The structure of the paper is organized in four sectios and a brief appendix. In Section 2 we find conditions about the born of a limit cycle through a Hopf bifurcation with $\gamma$ as the bifurcation parameter. In Section 3 we study the problem of finding analytical approximations of the periodic solution of model (\ref{DDE_GM_around_0}) using the Poincar\'e-Lindstedt series as a Fourier-Taylor series, whose orders are fully determined by manipulating power series algebraically (as is the case of automatic differentiation, see \cite{haro_canadell_figueras_luque_mondelo}). This algebraic manipulation allow us to find the periodic solution by solving order-by-order for the coefficients of the Poincar\'e-Lindstedt series by a linear recurrence equations. This leads in a practical way for the implementation of this section in any programming language. Moreover, without using any symbolic manipulator. In section 4, the model (\ref{DDE_GM_around_0}) is discretized by a collocation method, searching for points that lie on the periodic orbit. This transforms the problem into a Newton method, whose initial guess is precisely the Poincar\'e-Lindstedt series at order $\varepsilon^{3}$. This order is sufficient for the convergence of the Newton method. The found approximated periodic solution is the initial point of the pseudo-arclenght continuation method.
\section{Hopf Bifurcation}\label{hopf_bif_section}
We establish in the following the existence of a positive value $\gamma_{0}$ such that when the parameter $\gamma$ exceeds $\gamma_{0}$, the system (\ref{DDE_GM_around_0}) generate a limit cycle type oscillatory behaviour.

The system (\ref{DDE_GM_around_0}) is equivalent to,
\begin{equation}\label{DDE_GM_en formato_f}
	x'(t)
	=
	\gamma f
	\left(x(t), x(t-1), x(t-s_{0})\right),
\end{equation}
where $x=\left(u,v\right)^{T} \in \mathbb{R}^{2}$, $f:\mathbb{R}^{2}\times\mathbb{R}^{2}\times\mathbb{R}^{2} \rightarrow \mathbb{R}^{2}$ and
$p=\left(p_{1}, p_{2}\right)^{T}$, $q=\left(q_{1}, q_{2}\right)^{T}$, $w=\left(w_{1}, w_{2}\right)^{T}\in\mathbb{R}^{2}$ as,
\begin{equation*}
	f(p, q, w)
	=
	\left(
	\begin{array}{c}
		\dfrac{\left( p_{1} + u_{0} \right)^{2}}{w_{2} + v_{0}} - b\left( p_{1} + u_{0} \right) + a \\ \\
     	\left( q_{1} + u_{0} \right)^{2} - \left( p_{2} + v_{0} \right) + c
	\end{array}
	\right).
\end{equation*}

We ommit $q_{2}$, $w_{1}$ in definition of $f$ because $v(t-1)$ and $u(t-s_{0})$ do not appear in (\ref{DDE_GM_around_0}).

\begin{remark}\label{remark_formal_definition_of_C_n_tau}
In general, a delay differential equation with finite discrete delays is of the form
\begin{equation}\label{general_autonomous_DDE}
	x'(t) = f(x(t), x(t-\tau_{1}),\dots,x(t-\tau_{K})),
\end{equation}
with $x(t)\in\mathbb{R}^{n}$ and $f:\mathbb{R}^{(K+1)n}\rightarrow\mathbb{R}^{n}$ has $0$ as equilibrium (system (\ref{DDE_GM_en formato_f}) is a particular case). Using the formalism of \cite{guo_wu}, let $\tau=\max{\tau_{k}}$ and assume that $\mathbb{R}^{n}$ is equipped with the Euclideam norm $\vert \cdot\vert$. For $t_{0}\in\mathbb{R}$ define the mapping $x_{t_{0}}$ by $x_{t_{0}}(\theta)=x(t_{0}+\theta)$ for $\theta\in\left[-\tau, 0\right]$. This function $x_{t_{0}}$, uniquely determines $x(t)$ for all $t\geq t_{0}$. In this form, the state space for (\ref{general_autonomous_DDE}) is $\mathcal{C}_{n,\tau} = \mathcal{C}\left(\left[-\tau, 0\right],\mathbb{R}^{n}\right)$, which denote the Banach space of continuous mappings from $\left[-\tau, 0\right]$ into $\mathbb{R}^{n}$ equiped with the supremum norm $\Vert \phi\Vert=\sup_{-\tau\leq \theta\leq 0}\vert \phi(\theta)\vert$ for $\phi\in\mathcal{C}_{n,\tau}$. Therefore if $F:\mathcal{C}_{n,\tau}\rightarrow\mathbb{R}^{n}$ is defined as $F(\phi) = f(\phi(0), \phi(-\tau_{1}),\dots, \phi(-\tau_{K}))$ then system (\ref{general_autonomous_DDE}) can be rewritten as 
\begin{equation*}
x'(t) = F(x_{t}),
\end{equation*}
whose linearization $\mathcal{L}:\mathcal{C}_{n,\tau}\rightarrow\mathbb{R}^{n}$, using the Riezs representation theorem, is given by
\begin{equation}\label{linearization_of_a_general_autonomous_DDE}
x'(t)=\mathcal{L}x_{t} = \displaystyle\int_{-\tau}^{0} d \eta(\theta)x_{t}(\theta),
\end{equation}
where $\eta:\left[-\tau,0\right]\rightarrow\mathbb{R}^{n^{2}}$ is an $n\times n$ matrix-valued function whose components are of bounded variation.

In contrast to $\mathbb{R}^{n}$, the space $\mathcal{C}_{n,\tau}$ does not have a natural inner product associated with its norm. However following [Hale], one can introduce a substitute device that acts like an inner product in $\mathcal{C}_{n,\tau}$, therefore is necesary to consider an adjoint operator. For this, let $\mathcal{C}_{n,\tau}^{*}=\mathcal{C}\left(\left[0,\tau\right],\mathbb{R}^{n*}\right)$ be the space of continuous functions from $\left[0,\tau\right]$ to $\mathbb{R}^{n*}$ with $\Vert\psi\Vert=\sup\limits_{t\in\left[0,\tau\right]}\vert\psi(t)\vert$ for $\psi\in\mathcal{C}_{n,\tau}^{*}$, where $\mathbb{R}^{n*}$ is the space of $n$-dimensional real row vectors. The formal adjoint equation associated with the linear equation (\ref{linearization_of_a_general_autonomous_DDE}) is given by
\begin{equation}\label{general_adjoint_operator}
\psi'(t)=-\displaystyle\int_{-\tau}^{0} \psi(t-\theta)d\eta(\theta),
\end{equation}
for $\psi\in\mathcal{C}_{n,\tau}^{*}$.
\end{remark}

To linealizate (\ref{DDE_GM_en formato_f}) calculate $A = D_{p}f(0)$, $B_{1} = D_{q}f(0)$ and $B_{2}=D_{w}f(0)$. Hence
\begin{equation*}
	A
	=
	\left(
	\begin{array}{cc}
		\dfrac{2 u_{0}}{v_{0}} - b &     0   \vspace{-2mm} \\ \\
     	0                          &     -1
	\end{array}
	\right),\text{ }
	B_{1}
	=
	\left(
	\begin{array}{cc}
		      0      & 0 \\ \\
     	   2 u_{0}   & 0
	\end{array}
	\right),\text{ }
	B_{2}
	=
	\left(
	\begin{array}{cc}
		0 & -\dfrac{u_{0}^{2}}{v_{0}^{2}} \vspace{-2mm} \\ \\
        0 &                    0
	\end{array}
	\right).
\end{equation*}

Then the linearization of (\ref{DDE_GM_en formato_f}) around $0$ is given by
\begin{equation*}
	\begin{split}
		x'(t)
		&=
		\gamma \displaystyle\int_{-\tau}^{0} d \eta(\theta)x_{t}(\theta) \\	\\
		&=
		\gamma\left[Ax\left(t\right) + B_{1}x\left(t-1\right) + B_{2}x\left(t-s_{0}\right) \right] \\	\\
		&=:
		\gamma L(x(t), x(t-1), x(t-s_{0})),
	\end{split}
\end{equation*}
where
\begin{equation}\label{definition_of_eta}
	\eta(\theta)
	=
	\left\{
			\begin{array}{ll} 
				B_{2} + B_{1} + A,	& \theta=0, \\ \\ 
				B_{2} + B_{1},		& \theta\in\left[-1,0\right), \\ \\ 
				B_{2},				& \theta\in\left(-s_{0},-1\right), \\ \\
				0,					& \theta = -s_{0}.
			\end{array}
	\right.
\end{equation}

According to \cite{guo_wu}, its characteristic matrix is
\begin{equation*}
	\begin{split}
		\Delta(\lambda,\gamma)
		&=
		\lambda I_{2} - \gamma\left( A + e^{-\lambda}B_{1} + e^{-\lambda s_{0}}B_{2} \right) \\
		&=
		\left(
		\begin{array}{cc}
			\lambda+\gamma\left(-\dfrac{2 u_{0}}{v_{0}}+b\right) & \gamma\dfrac{u_{0}^{2}}{v_{0}^{2}}e^{-\lambda s_{0}} \\ \\
        	-2\gamma u_{0}e^{-\lambda}                &            \lambda+\gamma
		\end{array}
		\right).
	\end{split}
\end{equation*}

Using the notation in \cite{cooke_grossman}, we write the characteristic equation,
\begin{equation}\label{DDE_characteristic_equation}
	M(\lambda, \gamma)
	=
	\lambda^{2} + \gamma b_{1}\lambda + \gamma^{2}b_{0} + \gamma^{2} b_{2} e^{-2\tau\lambda},
\end{equation}
where,
\begin{subequations}\label{eq:b0_b1_b2_tau_values}
	\begin{align} 
		b_{0}
		&=
		-\dfrac{2 u_{0}}{v_{0}}+b, \label{eq:b0_value} \\
		b_{1}
		&=
		-\dfrac{2 u_{0}}{v_{0}}+b+1, \label{eq:b1_value} \\
		b_{2}
		&=
		\dfrac{2 u_{0}^{3}}{v_{0}^{2}}, \label{eq:b2_value}
	\end{align} 
\end{subequations}
and $\tau = \dfrac{s_{0}+1}{2}.$ Notice that $b_{2}>0$. Let's analyze the zeros of $M$ of the form $\lambda = i\omega$ with $\omega, \gamma>0$. In this way $ M(i\omega, \gamma) = 0$ is equivalent to
\begin{subequations}\label{eq:real_part_imaginary_part_characteristic_equation}
	\begin{align} 
		-\omega^{2} + \gamma^{2} b_{0} + \gamma^{2} b_{2}\cos(2\omega\tau)
		&=
		0, \label{eq:re_part_characteristic_equation} \\ 
		\gamma b_{1}\omega - \gamma^{2} b_{2}\sin(2\omega\tau)
		&=
		0. \label{eq:im_part_characteristic_equation} 
	\end{align} 
\end{subequations}

Notice that from (\ref{eq:im_part_characteristic_equation}) $b_{1}=0$ if and only if $\sin\left(2\omega\tau\right)=0$. Therefore we consider two cases $b_{1}\neq 0$ and $b_{1}=0$.
\begin{enumerate}
	\item $b_{1}\neq 0$. From (\ref{eq:im_part_characteristic_equation}) it follows that
	\begin{equation}\label{gamma0_hopf_bifurcation_b1_neq_0}
		\gamma = \dfrac{b_{1}\omega}{b_{2}\sin\left(2\omega \tau\right)}.
	\end{equation}

	To find $\omega$, we substitute $\gamma$ in (\ref{eq:re_part_characteristic_equation}) and let $\Omega = \cos\left(2\omega\tau\right)$. We have $b_{2}^{2}\Omega^{2} + b_{1}^{2}b_{2}\Omega + \left(b_{0}b_{1}^{2} - b_{2}^{2} \right) = 0$, whose solutions are $\Omega_{\pm} = \dfrac{-b_{1}^{2} \pm \sqrt{\delta}}{2b_{2}}$ and its discriminant is $\delta = b_{1}^{2}\left(b_{1}^{2}-4b_{0}\right) + 4b_{2}^{2}$. Notice from
	\begin{equation}\label{b1_at_2_minus_4b0}
		b_{1}^{2} - 4b_{0} = \left(b_{1} - 2\right)^{2} \geq 0,
	\end{equation}
	that $\delta\geq 0$. The following lemma shows that only $\Omega_{+}$ is plausible.
	\begin{lemma}
		$\vert \Omega_{-}\vert >1$
	\end{lemma}
	
	\begin{proof}
		Looking for a contradiction, we suppose that $\vert \Omega_{-}\vert \leq 1$ and arrive to,
		\begin{equation*}
			b_{1}^{2} + \sqrt{\delta} - 2b_{0} \leq 0.
		\end{equation*}

		Since
		\begin{equation}\label{b1_at_2_minus_2b0}
			b_{1}^{2} - 2b_{0} = \left( b_{1} -1 \right)^{2} + 1 > 0,
		\end{equation}
		then $b_{1}^{2} - 2b_{0} \leq -\sqrt{\delta}$, which is not possible.
	\end{proof}

	 Following \cite{cooke_grossman}, we look for $\vert\Omega_{+}\vert < 1$ in the three cases regarding the sign of $b_{0}^{2} - b_{2}^{2} < 0$.
	\begin{enumerate}
		\item $b_{0}^{2} - b_{2}^{2} < 0$. Adding $\left(b_{1}^{2} - 2b_{0} \right)^{2}$ on both sides of the expression $0 < -4\left( b_{0}^{2} - b_{2}^{2} \right)$ and taking the square root implies $b_{1}^{2} - 2b_{0} < \sqrt{\delta}$. Multiplying by $2b_{1}^{2}$, adding $4b_{2}^{2}$ on both sides, we have $\left(\dfrac{ b_{1} - \sqrt{\delta}}{2b_{2}}\right)^{2} < 1$. Taking the square root obtain $\vert \Omega_{+}\vert < 1$. In this way, 
\begin{equation}\label{omega0_hopf_bifurcation_b1_neq_0}
	\omega = \dfrac{\arccos\left(\Omega_{+}\right) + 2\pi j}{2\tau},
\end{equation}
with $j\in\mathbb{Z}$ and choosing $\arccos$ so that $\arccos\left(\Omega_{+}\right) + 2\pi j>0$. Using (\ref{gamma0_hopf_bifurcation_b1_neq_0}) notice that $b_{1}$ and $\sin\left(2\omega\tau \right)$ have the same sign. Therefore the branch of $\arccos$ depends of the sign of $b_{1}$.

		\item $b_{0}^{2} - b_{2}^{2} = 0$. In this case $\delta = \left( b_{1}^{2} - 2b_{0} \right)^{2}$, taking the square root and using (\ref{b1_at_2_minus_2b0}) we obtain $-b_{1}^{2} + \sqrt{\delta} = -2b_{0}$. By the hypothesis $b_{0}^{2}=b_{2}^{2}$ then $b_{0}\neq 0$. Thus $\dfrac{-b_{1}^{2} + \sqrt{\delta}}{2b_{0}} = -1$, which implies that $\Omega_{+} = \pm 1$, so $2\omega\tau = j\pi$ with $j\in\mathbb{N}$ and $\sin\left(2\omega\tau_{0}\right) = \sin(j\pi) = 0$. Using (\ref{eq:im_part_characteristic_equation}) we have $b_{1}=0$ which is not possible.
					
		\item $b_{0}^{2} - b_{2}^{2} > 0$. Adding $\left( b_{1}^{2} - 2b_{0} \right)^{2}$ on both sides of the expression $-4\left( b_{0}^{2} - b_{2}^{2}\right)<0$ and taking the square root implies $0 <b_{1}^{2} - 2b_{0} -\sqrt{\delta}$. Multiplying by $2b_{1}^{2}$ and adding $4b_{2}^{2}$ on both sides, we have $1 < \left(\dfrac{-b_{1}^{2} + \sqrt{\delta}}{2b_{2}}\right)^{2}$. Taking the square root we obtain $1 < \vert \Omega_{+}\vert$. This shows that in this case $\Omega_{+}$ does not lead to solutions.
	\end{enumerate}

	\item Case 2. $b_{1} = 0$. We have that $\cos(2\tau\omega) = \pm 1$. By (\ref{eq:re_part_characteristic_equation}) and if $b_{2}>1$ then for $k\in\mathbb{N}$,
		\begin{equation}\label{gamma0_omega0_hopf_bifurcation_b1_equal_0}
			\gamma = \dfrac{k\pi}{\tau\sqrt{b_{2}-1}}, \hspace{2mm} \omega = \dfrac{k\pi}{\tau}.		
		\end{equation}
\end{enumerate}

By implicit differentiation of $M$ with respect to $\gamma$ we obtain,
\begin{equation*}
	\dfrac{d\lambda}{d\gamma} = \dfrac{2\lambda^{2} + \gamma b_{1}\lambda}{2\gamma\tau\lambda^{2} + \left( 2\gamma + 2\gamma^{2}\tau b_{1}\right)\lambda + \left(\gamma^{2}b_{1} + 2\gamma^{3}\tau b_{0}\right)}.
\end{equation*}

Let $i\omega_{0}$ be a root of $M(\lambda,\gamma)=0$ associated with $\gamma=\gamma_{0}$. Observe that the denominator of the above expression is a cuadratic equation in the variable $\lambda$ with real coefficients, whose discriminant is $4\gamma^{2} + 4\gamma^{4}\tau^{2}\left(b_{0}^{2} - 1\right)^{2} > 0$, therefore $\lambda=i\omega_{0}$ is not a root of the denominator when $\gamma=\gamma_{0}$. Using (\ref{b1_at_2_minus_2b0}) we have,
\begin{equation}\label{Hopf_transversality}
\operatorname{Re}\left(\dfrac{d\lambda}{d\gamma}\right)\Bigg\vert\substack{\vspace{0mm}\lambda=i\omega_{0},\\	\hspace{-2mm}\vspace{-3mm}\gamma=\gamma_{0}\vspace{-3mm}}
=
\dfrac{4\tau\gamma_{0}\omega^{4}_{0} + 2\tau\gamma^{3}_{0}\omega^{2}_{0}\left( b_{1}^{2} - 2b_{0} \right)}{\left(\gamma^{2}_{0}b_{1} - 2\gamma_{0}\tau\omega^{2}_{0} + 2\gamma^{3}_{0}\tau b_{0} \right)^{2} + \left(2\gamma_{0}\omega_{0} + 2\gamma^{2}_{0}\tau b_{1}\omega_{0} \right)^{2}} > 0.
\end{equation}

We will show that $\lambda = i\omega_{0}$ is a simple root of $M\left(\lambda, \gamma \right)=0$ associated with $\gamma=\gamma_{0}$. Suppose that $\lambda_{0}$ is a zero of multiplicity $k+1$, with $k\geq 1$. Then there is $g_{0}$ analytic on a neihborhood $\mathcal{U}\subset\mathbb{C}$ of $\lambda_{0}$ such that $g_{0}\left(\lambda_{0} \right)\neq 0$ and $M\left(\lambda,\gamma_{0}\right) = \left(\lambda - \lambda_{0} \right)^{k+1} g_{0}\left(\lambda\right)$ with $\lambda\in\mathcal{U}$ then,
\begin{equation*}
\begin{split}
0
&=
\dfrac{dM\left(\lambda,\gamma_{0} \right)}{d\lambda}\Bigg\vert_{\lambda=\lambda_{0}} \\
&=
2\tau\lambda_{0}^{2} + 2\left(1 + \tau\gamma_{0}b_{1}\right)\lambda_{0} + \gamma_{0}\left(b1 + 2\tau\gamma_{0}b_{0} \right).
\end{split}
\end{equation*}

This quadratic equation has real coefficients. Using (\ref{b1_at_2_minus_4b0}) the discriminant is $4 + 4\tau^{2}\gamma_{0}^{2}\left(b_{1}^{2} - 4b_{0} \right)>0$ and thus $\lambda_{0}\in\mathbb{R}$, which is not possible. Therefore $\lambda = i\omega_{0}$ is a simple root of $M\left(\lambda,\gamma_{0}\right) = 0$.

We notice that $\lambda = n i\omega_{0}$ is not a root of $M\left(\lambda,\gamma_{0} \right) = 0$ for all $n\in\mathbb{Z}\setminus\left\lbrace\pm 1 \right\rbrace$. Observe that $i\omega_{0}$ and $n i\omega_{0}$ satisfies $\gamma_{0}^{4}b_{2}^{2} = \left(\omega_{0}^{2} - \gamma_{0}^{2}b_{0}\right)^{2} + \gamma_{0}^{2}b_{1}^{2}\omega_{0}^{2}$ and $\gamma_{0}^{4}b_{2}^{2} = \left(n^{2}\omega_{0}^{2} - \gamma_{0}^{2}b_{0}\right)^{2} + \gamma_{0}^{2}b_{1}^{2}n^{2}\omega_{0}^{2}$. Combining these two expressions we obtain $\left(n^{2} + 1 \right)\omega_{0}^{2} + \gamma_{0}^{2}\left(b_{1}^{2} - 2b_{0} \right) = 0$. However, we have that $\gamma_{0}^{2}\left(b_{1}^{2} - 2b_{0} \right)>0$ and $\left(n^{2} + 1 \right)\omega_{0}^{2}>0$. This contradicts the previous equation. Therefore we can formulate the following proposition.
\begin{proposition}\label{Hopf_bifurcation_theorem}
If $b_{1}\neq 0$ and $b_{0}^{2} - b_{2}^{2} < 0$, then system (\ref{DDE_GM_around_0}) undergoes a Hopf bifurcation, whose bifurcation values are $\gamma_{0}$ given by (\ref{gamma0_hopf_bifurcation_b1_neq_0}) and $\omega_{0}$ given by (\ref{omega0_hopf_bifurcation_b1_neq_0}). If $b_{1}=0$ and $b_{2}>1$ then system (\ref{DDE_GM_around_0}) undergoes a Hopf bifurcation, whose bifurcation values $\gamma_{0}$ and $\omega_{0}$ are given by (\ref{gamma0_omega0_hopf_bifurcation_b1_equal_0}). The transversality condition required is given by (\ref{Hopf_transversality}). Moreover, the root is a simple one and no other root is an integer multiple of $i\omega_{0}$.
\end{proposition}
\section{Poincar\'e - Lindstedt Series for a periodic solution}
We proceed to construct an approximation to the bifurcating periodic solution using the Poincar\'e-Lindstedt series. The idea is to develop order-by-order calculations for the coefficients solving linear recursive equations.

Given that we are interested in finding an analytic periodic solution of (\ref{DDE_GM_en formato_f}), suppose that this system has a periodic solution $x$ of period $T>0$ and frecuency $\omega = \frac{2\pi}{T}$. Defining $y(t) = x\left(\frac{t}{\omega}\right)$ then $x$ is $T$-periodic if and only if $y$ is $2\pi$-periodic and system (\ref{DDE_GM_en formato_f}) is equivalent to
\begin{equation}\label{DDE_GM_en_formato_f_delays_w_s0w}
\omega y'(t) = 
\gamma
f
\left(
y\left(t\right),
y\left(t-\omega\right), 
y\left(t-s_{0}\omega\right)
\right).
\end{equation}
Since periodic solutions of analytic delay differential equations are analytics \cite{nussbaum_73}, we look for the solution of (\ref{DDE_GM_en_formato_f_delays_w_s0w}) in the form of a perturbative series,
\begin{equation}\label{y_U_V_perturbed}
	y(t,\varepsilon)
	:=
	\begin{pmatrix}
		U(t,\varepsilon)	\\
		V(t,\varepsilon)
	\end{pmatrix}
	=
	\sum\limits_{k=0}^{\infty}
	\begin{pmatrix}
		U_{k}(t)	\\
		V_{k}(t)
	\end{pmatrix}
	\varepsilon^{k}
	=:
	\sum\limits_{k=0}^{\infty} y_{k}(t)\varepsilon^{k},
\end{equation}
where $\varepsilon$ is a small positive number and $U(\cdot,\varepsilon)$, $V(\cdot,\varepsilon)$ are $2\pi$-periodic analytic functions. We observe that $y(\cdot,\varepsilon)$ (which is $2\pi$-periodic) belongs to $\mathcal{C}_{2, s_{0}\omega}$ (see \textbf{Remark} \textbf{\ref{remark_formal_definition_of_C_n_tau}}) and its condition of being a periodic function allows us to obtain the equality of the \emph{initial function segment} $y(\cdot, \varepsilon)$ and the \emph{final function segment} $y(2\pi+\cdot, \varepsilon)$, i.e. 
$y(\theta, \varepsilon) = y(2\pi + \theta, \varepsilon)$ for $\theta\in\left[-s_{0}\omega, 0\right]$, which is a periodicity condition. The periodic solutions of the nonlinear equation (\ref{DDE_GM_en_formato_f_delays_w_s0w}) have periods depending on the parameters $\gamma$ and $\omega$. Hence we perturb both the parameter $\gamma$ and the frequency $\omega$,
\begin{equation}\label{gamma_and_omega_perturbed}
	\gamma(\varepsilon) = \sum\limits_{k=0}^{\infty} \gamma_{k}\varepsilon^{k},\hspace{2mm} \omega(\varepsilon) = \sum\limits_{k=0}^{\infty} \omega_{k}\varepsilon^{k}.
\end{equation}
Substituting these expansions into system (\ref{DDE_GM_en_formato_f_delays_w_s0w}) and using (\ref{y_U_V_perturbed}) and (\ref{gamma_and_omega_perturbed}), we obtain a system of the form, 
\begin{equation}\label{DDE_GM_en_formato_f_delays_w_s0w_perturbed}
\omega(\varepsilon)\partial_{t}\left(y(t,\varepsilon) \right)
=
\gamma(\varepsilon)f\left(y(t,\varepsilon), y\left(t-\omega(\varepsilon),\varepsilon \right), y\left(t-s_{0}\omega(\varepsilon),\varepsilon \right) \right).
\end{equation}

In order to find a periodic solution of system (\ref{DDE_GM_en_formato_f_delays_w_s0w_perturbed}) we formulate the following proposition that allows us to approximate  $y(\cdot, \varepsilon)$. The main idea is to solve (\ref{DDE_GM_en_formato_f_delays_w_s0w_perturbed}) expanding in powers of $\varepsilon^{k}$, $k\geq 0$, and equating the coefficients of the same power. The proposition contains some statements whose proofs are included.
%
\begin{proposition}\label{proposition_for_finding_a_periodic_solution_at_each_order_epsilon_power_k}
Equation (\ref{DDE_GM_en_formato_f_delays_w_s0w_perturbed}) has a $2\pi$-periodic solution for each order $\varepsilon^{k}$, for $k\in\mathbb{N}$.
\end{proposition}

\begin{proof}

At order $\varepsilon^{0}$, given that (\ref{DDE_GM_around_0}) is centered at the origin, then $y_{0}=0$, where $\gamma_{0}$ and $\omega_{0}$ are the values at the Hopf Bifurcation in Proposition \ref{Hopf_bifurcation_theorem}.

At order $\varepsilon^{1}$ we have,
\begin{equation}\label{DDE_GM_en_formato_f_delays_w_s0w_perturbed_at_order_epsilon_1}  
\omega_{0}y_{1}'(t) - \gamma_{0}L\left(y_{1}(t), y_{1}\left(t-\omega_{0}\right), y_{1}\left(t-s_{0}\omega_{0}\right)\right)
=
0,
\end{equation}
which is the linealization of (\ref{DDE_GM_en_formato_f_delays_w_s0w_perturbed}) around $0$. Expanding $y_{1}$ in Fourier series yeilds,
\begin{equation*}
	y_{1}(t)
	:=
	\begin{pmatrix}
		U_{1}(t)	\\
		V_{1}(t)
	\end{pmatrix}
	=
	\sum\limits_{n\in\mathbb{Z}}
	\begin{pmatrix}
		\hat{U}_{1}(n)	\\
		\hat{V}_{1}(n)
	\end{pmatrix}
	e^{i n t}
	=:
	\sum\limits_{n\in\mathbb{Z}}\hat{y}_{1}(n)e^{i n t},
\end{equation*}

with $\hat{y}_{1}(n) = \operatorname{conj}\left(\hat{y}_{1}(-n) \right) \in\mathbb{C}^{2}$ for all $n\in\mathbb{Z}$, we obtain that (\ref{DDE_GM_en_formato_f_delays_w_s0w_perturbed_at_order_epsilon_1}) is equivalent to the following system for each $n\in\mathbb{Z}$,
\begin{equation*}
\Delta\left(n i\omega_{0}, \gamma_{0}\right)\hat{y}_{1}(n)
=
0.
\end{equation*}

We notice that $\det\Delta(n i \omega_{0},\gamma_{0})\neq 0$ for $\vert n\vert\neq 1$. Therefore $\hat{y}_{1}(n) = 0$ wherever $n$ is different from $\pm 1$. For $n = 1$ we observe that $\Delta (i\omega_{0},\gamma_{0})$ is singular, so $\hat{y_{1}}(1) \in \ker\Delta(i\omega_{0},\gamma_{0})$. For simplicity we choose, 
\begin{equation}\label{y_1_1_hat}
\hat{y}_{1}(1)
= 
\begin{pmatrix}
	i\omega_{0} + \gamma_{0}	\\	\\
	2\gamma_{0}u_{0} e^{-i\omega_{0}}
\end{pmatrix}.
\end{equation}

Following \textbf{Remark \ref{remark_formal_definition_of_C_n_tau}}, the formal adjoint equation associated with the linear equation (\ref{DDE_GM_en_formato_f_delays_w_s0w_perturbed_at_order_epsilon_1}) is given by
\begin{equation}\label{lineal_adjoint_equation}
\omega_{0}\psi'(t) + \gamma_{0}\left[\psi(t) A + \psi\left(t + \omega_{0}\right)B_{1} + \psi\left(t + s_{0}\omega_{0}\right)B_{2}\right]
=
0,
\end{equation}
for $\psi(t)\in\mathbb{R}^{2*}$. Developing $\psi$ in Fourier series
$
\psi(t)
=
\sum\limits_{n\in\mathbb{Z}}\hat{\psi}(n)e^{i n t}
$
of (\ref{lineal_adjoint_equation}), with $\hat{\psi}(n) = \operatorname{conj}\left(\hat{\psi}(-n)\right)\in\mathbb{C}^{2*}$ for all $n\in\mathbb{Z}$. we obtain that (\ref{lineal_adjoint_equation}) is equivalent to solve the following system for each $n\in\mathbb{Z}$ 
\begin{equation*}
\hat{\psi}(n)\Delta(-n i\omega_{0},\gamma_{0})
=
0,
\end{equation*}
and thus $\hat{\psi}(n) = 0$ for all $\lvert n \rvert \neq 1$. For $n = -1$ just choose $\hat{\psi}(-1) \in \ker^{*}\Delta(i\omega_{0},\gamma_{0})$. For simplicity we choose
\begin{equation*}
\hat{\psi}(-1)
=
\begin{pmatrix}
i\omega_{0}+\gamma_{0},	&
-\gamma_{0}\dfrac{u_{0}^{2}}{v_{0}^{2}}e^{-i\omega_{0}s_{0}}
\end{pmatrix}.
\end{equation*}

We consider the following lemma
\begin{lemma}\label{Delta_minus_1_tends_to_0}
	\begin{equation*}
		\lim\limits_{\vert n\vert \rightarrow\infty}\big\Vert \Delta^{-1}\left(n i \omega_{0},\gamma_{0}\right)\big\Vert_{F} = 0,
	\end{equation*}
	where $\Vert \cdot \Vert_{F}$ is the Frobenius norm.
\end{lemma}

\begin{proof}
	Notice that for $\vert n\vert \neq 1$ we have
	\begin{equation*}
		\big\Vert \Delta^{-1}\left(n i \omega_{0},\gamma_{0}\right)\big\Vert^{2}_{F}
		=
		\dfrac{2\omega_{0}^{2}n^{2}+r_{0}}{\omega_{0}^{4}n^{4} + a_{2}(n)n^{2} + a_{1}(n)n + a_{0}(n)},
	\end{equation*}
	where
	\begin{equation*}
		\left.
			\begin{array}{rcl}
				a_{2}(n) 
				& = 
				& \gamma_{0}^{2}b_{1}^{2}\omega_{0}^{2} - 2\gamma_{0}^{2}\omega_{0}^{2}\left[b_{0} + b_{2}\cos\left(2\tau\omega_{0}n\right) \right], \\ \\
				a_{1}(n) 
				& = 
				& -2\gamma_{0}^{3}b_{1}b_{2}\omega_{0}\sin\left( 2\tau\omega_{0}n \right), \\ \\
				a_{0}(n)
				& =
				& \gamma_{0}^{4}b_{2}^{2}\sin^{2}\left( 2\tau\omega_{0}n \right) + \gamma_{0}^{4}\left[b_{0} + b_{2}\cos\left(2\tau\omega_{0}n\right) \right]^{2}, \\ \\
				r_{0}
				& =
				& \gamma_{0}^{2} + \gamma_{0}^{2}\dfrac{u^{4}}{v^{4}} + 4\gamma_{0}^{2}u_{0}^{2} + \gamma_{0}^{2}b_{0}^{2}.
			\end{array}
		\right.
	\end{equation*}
	
	We observe that sequences $a_{2}(n), a_{1}(n)$ and $a_{0}(n)$ are bounded. Given that $\big\Vert \Delta^{-1}\left(-n i \omega_{0},\gamma_{0}\right)\big\Vert_{F} = \big\Vert \Delta^{-1}\left(n i \omega_{0},\gamma_{0}\right)\big\Vert_{F}$, then without loss of generality we suppose that $n\in\mathbb{N}$. In this form, let $\varepsilon>0$. By the Archimedean property there is $n_{1}\in\mathbb{N}$ such that for $n\geq n_{1}$ we have $2\omega_{0}^{2}n^{2} + r_{0} < \varepsilon\omega_{0}^{4}n^{4} + \varepsilon\left[ a_{2}(n)n^{2} + a_{1}(n)n + a_{0}(n) \right]$, which is equivalent to
	\begin{equation*}
		\big\Vert \Delta^{-1}\left(n i \omega_{0},\gamma_{0}\right)\big\Vert_{F} < \sqrt{\varepsilon}.
	\end{equation*}
\end{proof}

Considering the operator $\Pi_{k}:\mathcal{C}^{1}\left(\mathbb{R},\mathbb{R}^{2}\right)\rightarrow\mathcal{C}\left(\mathbb{R},\mathbb{R}^{2}\right)$, $k\in\mathbb{N}_{0}$, defined by (see \cite{guo_wu})
\begin{equation*}
	\left(\Pi_{k}\varphi\right)(t)
	=
	\omega_{k}\varphi'(t) - \gamma_{k}L(\varphi(t), \varphi(t-\omega_{0}), \varphi(t - s_{0}\omega_{0})),
\end{equation*}
and for simplicity $\Pi = \Pi_{0}$, we have the following proposition (see for example \cite{gopalsamy1996395}, \cite{hale_verduyn93}, \cite{hale71})
\begin{proposition}\label{proposition_for_finding_periodic_solutions}
Consider the equation
	\begin{equation}\label{Pi_eq_R}
		\left(\Pi\varphi\right)(t)
		=
		R(t),
	\end{equation}
	where $R$ is real $2\pi$-periodic and expanding on Fourier series, $R(t) = \sum\limits_{n\in\mathbb{Z}}\hat{R}(n)e^{i n t}$. The following statemets are equivalent
	\begin{itemize}
		\item[(a)] $\hat{R}(1)\in \operatorname{Range}\Delta\left(i\omega_{0}, \gamma_{0}\right)$.
		\item[(b)] Equation (\ref{Pi_eq_R}) has at least one real non constant $2\pi$-periodic solution.
		\item[(c)] $\hat{\psi}(-1)\hat{R}(1) = 0$.
	\end{itemize}
\end{proposition}

\begin{proof}
We prove $(a)\Rightarrow(b)\Rightarrow(c)\Rightarrow(a)$.

$(a)\Rightarrow(b)$ Looking for a real $2\pi$-periodic solution, we consider the \textit{ansatz} $\varphi(t) = \sum\limits_{n\in\mathbb{Z}} \hat{\varphi}(n) e^{int} $, $\hat{\varphi}(n) = \operatorname{conj}\left(\hat{\varphi}(-n)\right) \in\mathbb{C}^{2}$ for all $n\in\mathbb{Z}$. Then equation (\ref{Pi_eq_R}) implies solve for each $n\in\mathbb{Z}$
\begin{equation*}
\Delta\left(ni\omega_{0},\gamma_{0}\right)\hat{\varphi}(n) = \hat{R}(n),
\end{equation*}
obtaining that $\hat{\varphi}(n) = \Delta^{-1}\left(ni\omega_{0},\gamma_{0}\right)\hat{R}(n)$ for $\lvert n \rvert \neq 1$. For $n=1$, given that $ \dim\left(\ker\Delta\left(i\omega_{0},\gamma_{0}\right)\right) = 1$ then $\operatorname{Range}\Delta\left(i\omega_{0},\gamma_{0}\right) = \operatorname{span}_{\mathbb{C}}\left\lbrace \Delta\left(i\omega_{0},\gamma_{0}\right)\textbf{e}_{1} \right\rbrace $. By hypothesis there is $c_{1}\in\mathbb{C}$ such that $\hat{R}(1) = c_{1}\Delta\left(i\omega_{0},\gamma_{0}\right)\textbf{e}_{1}$, then we choose $\hat{\varphi}(1) = c_{1}\textbf{e}_{1}$. By Lemma \ref{Delta_minus_1_tends_to_0} we have that $\big\Vert\Delta^{-1}\left(n i \omega_{0}, \gamma_{0} \right)\big\Vert_{F}$ is bounded, then $\hat{\varphi}(n)$ decreases exponentially, obtaining at least one real non constant $2\pi$-periodic solution.

$(b)\Rightarrow(c)$ Suppose that $\varphi(t) = \sum\limits_{n\in\mathbb{Z}} \hat{\varphi}(n) e^{int}$ with $\hat{\varphi}(n) = \operatorname{conj}\left(\hat{\varphi}(-n)\right) \in\mathbb{C}^{2}$ for all $n\in\mathbb{Z}$, is a real non constant $2\pi$-periodic solution of (\ref{Pi_eq_R}), then for $n=1$ we have $\Delta\left(i\omega_{0},\gamma_{0}\right)\hat{\varphi}(1) = \hat{R}(1)$. Left multiplying by $\hat{\psi}(-1)$ we get 
\begin{equation*}
\hat{\psi}(-1) \hat{R}(1)
=
\hat{\psi}(-1) \Delta\left(i\omega_{0},\gamma_{0}\right)\hat{\varphi}(1) \\
=
0.
\end{equation*}

$(c)\Rightarrow(a)$ $\hat{\psi}(-1)
\begin{pmatrix}
	\hat{R}^{(1)}(1)	\\
	\hat{R}^{(2)}(1)
\end{pmatrix}=0$ implies that $\hat{R}^{(2)}(1) = \dfrac{i\omega_{0} + \gamma_{0}}{\gamma_{0}\frac{u^{2}_{0}}{v^{2}_{0}} e^{-i\omega_{0}s_{0}}} \hat{R}^{(1)}(1)$. Thus
\begin{equation*}
\hat{R}(1)\in \operatorname{span}_{\mathbb{C}}
			\left\lbrace
 			\begin{pmatrix}
				1	\\	\\
         		\dfrac{i\omega_{0} + \gamma_{0}}{\gamma_{0}\frac{u^{2}_{0}}{v^{2}_{0}} e^{-i\omega_{0}s_{0}}}
			\end{pmatrix}
			\right\rbrace
			=
			\operatorname{span}_{\mathbb{C}}
			\left\lbrace
 			\begin{pmatrix}
				\gamma_{0}\frac{u^{2}_{0}}{v^{2}_{0}} e^{-i\omega_{0}s_{0}}	\\	\\
         		i\omega_{0} + \gamma_{0}
			\end{pmatrix}
			\right\rbrace
			=
			\operatorname{span}_{\mathbb{C}}
			\left\lbrace
				\Delta\left(i\omega_{0},\gamma_{0} \right)\textbf{e}_{2} 			
 			\right\rbrace
 			=
 			\operatorname{Imag}\Delta\left(i\omega_{0},\gamma_{0} \right).
\end{equation*}
\end{proof}

At order $\varepsilon^{2}$ we have,
\begin{equation}\label{PerSol_at_order_epsilon_2}
\left(\Pi y_{2}\right)(t)
=
R_{2}(t),
\end{equation}
where $R_{2}(t)
	=
	-
	\left\lbrace\left(\Pi_{1}y_{1}\right)(t) + \gamma_{0}\omega_{1}\left[B_{1}y_{1}'\left(t - \omega_{0}\right) + B_{2}y_{1}'\left(t - s_{0}\omega_{0}\right) \right]\right\rbrace
	+
	\gamma_{0}G_{2}(t)
$, $G_{2}$ is defined by $G_{2}(t) = \left(G^{(1)}_{2}(t), G^{(2)}_{2}(t)\right)^{T}$ with $G^{(1)}_{2}(t) = \dfrac{1}{v_{0}}\left[U_{1}(t) - \dfrac{u_{0}}{v_{0}}V_{1}\left(t - s_{0}\omega_{0}\right)\right]^{2}$ and $G^{(2)}_{2}(t) = U_{1}^{2}\left(t - \omega_{0}\right)$.

Expanding (\ref{PerSol_at_order_epsilon_2}) in Fourier series, we have that for each $n\in\mathbb{Z}$,
\begin{equation*}
\Delta\left(n i\omega_{0}, \gamma_{0}\right)\hat{y}_{2}(n)
=
\hat{R}_{2}(n),
\end{equation*}
with $\hat{R}_{2}(n)=
	-
	\left[\tilde{\Delta}\left(n i\omega_{1}, \gamma_{1}, n i\omega_{0}\right) + \gamma_{0}n i\omega_{1}\left(e^{-n i\omega_{0}}B_{1} + s_{0}e^{-n i\omega_{0}s_{0}}B_{2}\right)\right]\hat{y}_{1}(n)
	+
	\gamma_{0}
	\hat{G}_{2}(n)
	$, where $\hat{G}_{2}(n)$ is defined by $\hat{G}_{2}(n)=\left(\hat{G}^{(1)}_{2}(n), \hat{G}^{(2)}_{2}(n)\right)^{T}$. Here $\hat{G}^{(1)}_{2}(n)$ and $\hat{G}^{(2)}_{2}(n)$ are written as follows,
\begin{equation*}
\hat{G}^{(1)}_{2}(n) = \dfrac{1}{v_{0}}\sum\limits_{\substack{n_{1} + n_{2} = n\\ n_{1,2}\in\mathbb{Z}}} \left[\hat{U}_{1}(n_{1}) - \frac{u_{0}}{v_{0}}\hat{V}_{1}(n_{1}) e^{-n_{1}i\omega_{0}s_{0}} \right] \left[\hat{U}_{1}(n_{2}) - \frac{u_{0}}{v_{0}}\hat{V}_{1}(n_{2}) e^{-n_{2}i\omega_{0}s_{0}} \right],
\end{equation*}	
and $\hat{G}^{(2)}_{2}(n)
=
e^{-n i\omega_{0}}\sum\limits_{\substack{n_{1} + n_{2} = n\\ n_{1,2}\in\mathbb{Z}}}
\hat{U}_{1}(n_{1})\hat{U}_{1}(n_{2})$.
In this form we observe that $\hat{R}_{2}(n) = 0$ for $\vert n\vert \geq 3$.
	
For $n = 2$ we obtain,
\begin{equation*}
\hat{R}_{2}(2)
=
\gamma_{0}
\begin{pmatrix}
	\dfrac{1}{v_{0}}\left[\hat{U}_{1}(1) - \frac{u_{0}}{v_{0}}\hat{V}_{1}(1)e^{-i\omega_{0}s_{0}}\right]^{2}\\ \\
	e^{-2i\omega_{0}}\hat{U}_{1}^{2}\left(1\right)
\end{pmatrix},
\end{equation*}
and for $n = 0$ we obtain,
\begin{equation*}
\hat{R}_{2}(0)
=
2\gamma_{0}
\begin{pmatrix}
	\dfrac{1}{v_{0}}\bigg\vert \hat{U}_{1}(1) - \dfrac{u_{0}}{v_{0}}\hat{V}_{1}(1)e^{-i\omega_{0}s_{0}}\bigg\vert^{2}	\\ \\
	\bigg\vert\hat{U}_{1}\left(1\right)\bigg\vert^{2}
\end{pmatrix}.
\end{equation*}

Therefore, $\hat{y}_{2}(n) = 0$ for $\vert n\vert\geq 3$ and $\hat{y}_{2}(n)=\Delta^{-1}(ni\omega_{0},\gamma_{0})\hat{R}_{2}(n)$ for $n \in \left\lbrace 0, 2\right\rbrace$.

For $n = 1$, we observe that $\hat{\psi}(-1) \hat{R}_{2}(1)=0$ implies that,
\begin{equation*}
\begin{split}
0
&=
\hat{\psi}(-1) \left[\tilde{\Delta}\left(i\omega_{1}, \gamma_{1}, i\omega_{0}\right) + \gamma_{0}i\omega_{1}\left(e^{-i\omega_{0}}B_{1} + s_{0}e^{-i\omega_{0}s_{0}}B_{2}\right)\right]\hat{y}_{1}(1) \\
&=
\left[ \left(b_{1} + 2\right)\omega^{2}_{0} + i\left(\dfrac{2\omega^{2}_{0}}{\gamma_{0}} - b_{1}\gamma_{0}\right)\omega_{0}\right]\gamma_{1} \\
&\hspace{4mm}+
\left\lbrace -\biggl[\left(b_{1} + 2\right)\gamma_{0}\omega_{0} + \left(s_{0} + 1\right)\left[\left(b_{0} + b_{1}\right)\gamma^{2}_{0}\omega_{0} - \omega^{3}_{0}\right]\biggr]\right. \\
&\hspace{11mm}+
i \left. \biggl[\left(b_{1}\gamma^{2}_{0} - 2\omega_{0}^{2}\right) + \left(s_{0} + 1\right)\left[\left(-b_{1} - 1\right)\gamma_{0}\omega_{0}^{2} + b_{0}\gamma_{0}^{3}\right]\biggr]\right\rbrace \omega_{1},
\end{split}
\end{equation*}
which is equivalent to the linear system
\begin{equation}\label{C(g1_w1)_equal_0}
\left(
	\begin{array}{cc}
	\left(b_{1} + 2\right)\omega_{0}^{2}											& -\left\lbrace \left(b_{1} + 2\right)\gamma_{0}\omega_{0} + \left( s_{0} + 1\right)\left[\left(b_{0} + b_{1}\right)\gamma_{0}^{2}\omega_{0} - \omega_{0}^{3}\right] \right\rbrace   \\ \\
     \left(\dfrac{2\omega_{0}^{2}}{\gamma_{0}} - b_{1}\gamma_{0}\right)\omega_{0}	& b_{1}\gamma_{0}^{2} - 2\omega_{0}^{2} + \left(s_{0} + 1\right)\left[\left(-b_{1}-1\right)\gamma_{0}\omega_{0}^{2} + b_{0}\gamma_{0}^{3}\right]
	\end{array}
\right)
\left(
	\begin{array}{c}
		\gamma_{1} \\ \\ \\
		\omega_{1}
	\end{array}
\right)
=
0.
\end{equation}

Let $C$ be the matrix in (\ref{C(g1_w1)_equal_0}) and by the fact that
\begin{equation*}
\det C = -\omega_{0}^{2}\left(s_{0}+1\right)\left[\left(b_{0}^{2} + 3\right)\gamma_{0}\omega_{0}^{2} + \left(b_{0}^{2} + 1\right)\gamma_{0}^{3} + \dfrac{2\omega_{0}^{4}}{\gamma_{0}}\right] < 0,
\end{equation*}
we have that $\gamma_{1} = \omega_{1} = 0$ and thus $\hat{R}_{2}(1) = 0$. So, choosing $\hat{y}_{2}(1) = \hat{y}_{1}(1)$ and using the argument of Proposition \ref{proposition_for_finding_periodic_solutions}, the curve $y_{2}(t)$ exists.

At order $\varepsilon^{3}$ we have,
\begin{equation}\label{order_e^3}
\left(\Pi y_{3}\right)(t)= R_{3}(t),
\end{equation}
where
$R_{3}(t)
=
-\left\lbrace\left(\Pi_{2}y_{1}\right)(t)-\gamma_{0}\omega_{2}\left[B_{1}y_{1}'(t-\omega_{0}) + s_{0}B_{2}y_{1}'(t-s_{0}\omega_{0})\right]\right\rbrace + G_{3}(t)
$. We observe that $G_{3}(t)=\left(G^{(1)}_{3}(t), G^{(2)}_{3}(t)\right)^{T}$ is given by,
\begin{equation*}
\begin{split} 
G_{3}^{(1)}(t)
&=
\dfrac{2\gamma_{0}}{v_{0}}\left[U_{1}(t)-\frac{u_{0}}{v_{0}}V_{1}\left(t-s_{0}\omega_{0}\right)\right]\left[U_{2}(t) - \frac{u_{0}}{v_{0}}V_{2}\left(t - s_{0}\omega_{0}\right)\right]\\
&\hspace{4mm}-
\dfrac{\gamma_{0}}{v_{0}^{2}}V_{1}\left(t - s_{0}\omega_{0}\right)\left[U_{1}(t) - \frac{u_{0}}{v_{0}}V_{1}\left(t - s_{0}\omega_{0}\right)\right]^{2},
\end{split}
\end{equation*}	
and $G_{3}^{(2)}(t)
=
2\gamma_{0}
U_{1}(t-\omega_{0})U_{2}(t-\omega_{0})$. Then equation (\ref{order_e^3}) is equivalent to solving,
\begin{equation*}
\Delta(ni\omega_{0},\gamma_{0}) \hat{y}_{3}(n)=\hat{R}_{3}(n), \text{ for each $n\in\mathbb{Z}$,}
\end{equation*}
where
$\hat{R}_{3}(n)
=
- \left[ \tilde{\Delta} (ni\omega_{2},\gamma_{2},ni\omega_{0}) + \gamma_{0}ni\omega_{2} (e^{-ni\omega_{0}} B_{1} + s_{0} e^{-ni\omega_{0}s_{0}} B_{2}) \right] \hat{y}_{1}(n)
+
\hat{G}_{3}(n)
$, with $\hat{G}_{3}(n)$ is written as follows $\hat{G}_{3}(n) = \left(\hat{G}^{(1)}_{3}(n), \hat{G}^{(2)}_{3}(n)\right)^{T}$, given by,
\begin{equation*}
\begin{split}
G_{3}^{(1)}(n)
&=
\dfrac{2\gamma_{0}}{v_{0}}
\sum\limits_{\substack{n_{1} + n_{2} = n\\ n_{1,2}\in\mathbb{Z}}} \left[\hat{U}_{1}(n_{1}) - \dfrac{u_{0}}{v_{0}}\hat{V}_{1}(n_{1})e^{-n_{1}i\omega_{0}s_{0}}\right] \left[\hat{U}_{2}(n_{2}) - \dfrac{u_{0}}{v_{0}}\hat{V}_{2}(n_{2})e^{-n_{2}i\omega_{0}s_{0}}\right] \\
&\hspace{4mm}-
\dfrac{\gamma_{0}}{v_{0}^{2}}
\sum\limits_{\substack{n_{1} + n_{2} + n_{3} = n\\ n_{1,2,3}\in\mathbb{Z}}}
\hat{V}_{1}(n_{1})e^{-n_{1}i\omega_{0}s_{0}}
\left[\hat{U}_{1}(n_{2}) - \dfrac{u_{0}}{v_{0}}\hat{V}_{1}(n_{2})e^{-n_{2}i\omega_{0}s_{0}}\right]
\left[\hat{U}_{1}(n_{3}) - \dfrac{u_{0}}{v_{0}}\hat{V}_{1}(n_{3})e^{-n_{3}i\omega_{0}s_{0}}\right],
\end{split}
\end{equation*}
and $\hat{G}_{3}^{(2)}(n)
=
2\gamma_{0}e^{-ni\omega_{0}}
\sum\limits_{\substack{n_{1} + n_{2} = n\\ n_{1,2}\in\mathbb{Z}}}
\hat{U}_{1}(n_{1})\hat{U}_{2}(n_{2})
$. In this form we observe that $\hat{R}_{3}(n) = 0$ for $\vert n\vert \geq 4$.

For $n = 3$ we have that,
\begin{equation*}
\begin{split}
\hat{R}_{3}(3)
&=
2\gamma_{0}
\begin{pmatrix}
	\dfrac{1}{v_{0}} \left[\hat{U}_{1}(1) - \dfrac{u_{0}}{v_{0}}\hat{V}_{1}(1)e^{-i\omega_{0}s_{0}}\right]\left[\hat{U}_{2}(2) - \dfrac{u_{0}}{v_{0}}\hat{V}_{2}(2)e^{-2i\omega_{0}s_{0}}\right] \\ \\
	e^{-3 i \omega_{0}}\hat{U}_{1}(1)\hat{U}_{2}(2)
\end{pmatrix} \\
&\hspace{4mm}-
\dfrac{\gamma_{0}}{v_{0}^{2}}
\begin{pmatrix}
	\hat{V}_{1}(1)e^{-i \omega_{0}s_{0}}\left[\hat{U}_{1}(1) - \dfrac{u_{0}}{v_{0}}\hat{V}_{1}(1)e^{-i\omega_{0}s_{0}}\right]^{2} \\ \\
	0
\end{pmatrix}.
\end{split}
\end{equation*}

For $n = 2$ we obtain,
\begin{equation*}
\hat{R}_{3}(2) = 2\gamma_{0}
\begin{pmatrix}
	\dfrac{1}{v_{0}} \left[\hat{U}_{1}(1) - \frac{u_{0}}{v_{0}}\hat{V}_{1}(1)e^{-i\omega_{0}s_{0}}\right]\left[\hat{U}_{2}(1) - \frac{u_{0}}{v_{0}}\hat{V}_{2}(1)e^{-i\omega_{0}s_{0}}\right] \\ \\
	e^{-2 i \omega_{0}}\hat{U}_{1}(1)\hat{U}_{2}(1)
\end{pmatrix},
\end{equation*}
and for $n = 0$,
\begin{equation*}
\hat{R}_{3}(0)
=
4\gamma_{0}\operatorname{Re}
\begin{pmatrix}
	\dfrac{1}{v_{0}} \left[\hat{U}_{1}(1) - \frac{u_{0}}{v_{0}}\hat{V}_{1}(1)e^{-i\omega_{0}s_{0}}\right]\left[\hat{U}_{2}(-1) - \frac{u_{0}}{v_{0}}\hat{V}_{2}(-1)e^{i\omega_{0}s_{0}}\right] \\ \\
	\hat{U}_{1}(1)\hat{U}_{2}(-1)
\end{pmatrix}.
\end{equation*}

Therefore $\hat{y}_{3}(n) = 0$ for $\vert n\vert \geq 4$ and $\hat{y}_{3}(n) = \Delta^{-1}(n i\omega_{0}, \gamma_{0})\hat{R}_{3}(n)$ for $n\in\left\lbrace 0, 2, 3\right\rbrace$.

For $n = 1$,
\begin{equation*}
\hat{R}_{3}(1)
=
- \left[ \tilde{\Delta} (i\omega_{2},\gamma_{2},i\omega_{0}) + \gamma_{0}i\omega_{2} (e^{-i\omega_{0}} B_{1} + s_{0} e^{-i\omega_{0}s_{0}} B_{2}) \right] \hat{y}_{1}(1)
+
\hat{G}_{3}(1),
\end{equation*}
where
\begin{equation*}
\begin{split} 
\hat{G}_{3}^{(1)} (1)
&=
\dfrac{2\gamma_{0}}{v_{0}} \left\lbrace  \left[ \hat{U}_{1}(-1)-\frac{u_{0}}{v_{0}} \hat{V}_{1}(-1)  e^{i\omega_{0}s_{0}} \right]  \left[ \hat{U}_{2} (2) - \dfrac{u_{0}}{v_{0}} \hat{V}_{2} (2)  e^{-2 i\omega_{0}s_{0}} \right]  \right.\\
&\hspace{16mm}+
\left.
\left[ \hat{U}_{1} (1) - \dfrac{u_{0}}{v_{0}} \hat{V}_{1} (1)  e^{-i\omega_{0}s_{0}} \right] \left[ \hat{U}_{2} (0) -\dfrac{u_{0}}{v_{0}} \hat{V}_{2} (0)   \right] 
 \right\rbrace \\
&\hspace{5mm}-
\dfrac{\gamma_{0}}{v_{0}^{2}} \left\lbrace \hat{V}_{1}(-1)e^{i\omega_{0}s_{0}} \left[ \hat{U}_{1}(1) - \dfrac{u_{0}}{v_{0}} \hat{V}_{1}(1)e^{-i\omega_{0}s_{0}} \right]^{2} \right. \\
&\hspace{16mm}+
\left.
2 \hat{V}_{1}(1)e^{-i\omega_{0}s_{0}} \bigg\vert \hat{U}_{1}(1)-\dfrac{u_{0}}{v_{0}} \hat{V}_{1}(1)e^{-i\omega_{0}s_{0}} \bigg\vert^{2}
\right\rbrace,
\end{split}
\end{equation*}
and $\hat{G}_{3}^{(2)}(1) = 2\gamma_{0}e^{-i\omega_{0}} \left[ \hat{U}_{1}(-1) \hat{U}_{2}(2) + \hat{U}_{1}(1) \hat{U}_{2}(0)\right]$.

Therefore, $\hat{\psi}(-1)\hat{R}_{3}(1) = 0$ implies that,
\begin{equation*}
\begin{split}
0
&=
-\left\lbrace \left[ \left(b_{1} + 2\right)\omega_{0}^{2} + i\left(\dfrac{2\omega_{0}^{2}}{\gamma_{0}} - b_{1}\gamma_{0}\right)\omega_{0}\right]\gamma_{2} \right. \\
&\hspace{10mm}+
\left\lbrace -\bigg[\left(b_{1} + 2\right)\gamma_{0}\omega_{0} + \left(s_{0} + 1\right)\left[\left(b_{0} + b_{1}\right)\gamma_{0}^{2}\omega_{0} - \omega_{0}^{3}\right]\bigg]\right. \\
&\hspace{17mm}+
\left. \left. i\bigg[\left(b_{1}\gamma_{0}^{2} - 2\omega_{0}^{2}\right) + \left(s_{0} + 1\right)\left[\left(-b_{1} - 1\right)\gamma_{0}\omega_{0}^{2} + b_{0}\gamma_{0}^{3}\right]\bigg]\right\rbrace\omega_{2}\right\rbrace \\
&\hspace{4mm}+
\hat{\psi}(-1)\hat{G}_{3}(1),
\end{split}
\end{equation*}
which is equivalent to,
\begin{equation*}
\begin{pmatrix}
	\gamma_{2} \\ \\
	\omega_{2}
\end{pmatrix} 
=
C^{-1}
\begin{pmatrix}
	\operatorname{Re}\left(\hat{\psi}(-1)\hat{G}_{3}(1) \right) \vspace{1mm} \\ 
	\operatorname{Im}\left(\hat{\psi}(-1)\hat{G}_{3}(1) \right)
\end{pmatrix},
\end{equation*}
and by Proposition \ref{proposition_for_finding_periodic_solutions} since $\hat{R}_{3}(1)\in\operatorname{Range} \Delta(i\omega_{0},\gamma_{0})$, the curve $y_{3}(t)$ exists.

For at order $\varepsilon^{k}$ for $k\geq 4$ we consider the following. Taking $N = f - L$ then (\ref{DDE_GM_en_formato_f_delays_w_s0w_perturbed}) is equivalent to,
\begin{equation}\label{DDE_GM_en_formato_N_equal_f_minus_L_delays_w_s0w_perturbed}
	\begin{split}
		\omega(\varepsilon)\partial_{t}\left(y(t,\varepsilon) \right)
		&-
		\gamma(\varepsilon)L\left(y(t,\varepsilon), y\left(t-\omega(\varepsilon),\varepsilon \right), y\left(t-s_{0}\omega(\varepsilon),\varepsilon \right) \right)	\\
		&=
		\gamma(\varepsilon)N\left(y(t,\varepsilon), y\left(t-\omega(\varepsilon),\varepsilon \right), y\left(t-s_{0}\omega(\varepsilon),\varepsilon \right) \right).
	\end{split}
\end{equation}

The left side of (\ref{DDE_GM_en_formato_N_equal_f_minus_L_delays_w_s0w_perturbed}), at order $\varepsilon^{k}$ is given by,
\begin{equation}\label{linear_part_of_DDE_perturbed_at_order_epsilon_k}
\begin{split}
\left(\Pi y_{k}\right)(t)
&+
\left(\Pi_{k-1}y_{1}\right)(t)
-
\dfrac{1}{(k-1)!}\gamma_{0}\left[B_{1}\partial_{\varepsilon}^{k-1}\left(y_{1}\left(t - \omega\left(\varepsilon\right)\right) \right)\big\vert_{\varepsilon=0} + B_{2}\partial_{\varepsilon}^{k-1}\left(y_{1}\left(t - s_{0}\omega\left(\varepsilon\right)\right) \right)\big\vert_{\varepsilon=0}\right] \\
&+
\sum\limits_{k_{1} = 2}^{k-2}\left\lbrace\left(\Pi_{k_{1}}y_{k-k_{1}}\right)(t) \textcolor{white}{\dfrac{1}{k_{1}}}	\right.	\\
&\hspace{14mm}-
\left.	\dfrac{1}{k_{1}!}\gamma_{0}\left[ B_{1}\partial_{\varepsilon}^{k_{1}}\left(y_{k-k_{1}}\left(t - \omega\left(\varepsilon\right)\right) \right)\big\vert_{\varepsilon=0} + B_{2}\partial_{\varepsilon}^{k_{1}}\left(y_{k-k_{1}}\left(t - s_{0}\omega\left(\varepsilon\right)\right) \right)\big\vert_{\varepsilon=0}\right] \right\rbrace \\
&-
\sum\limits_{k_{1} = 2}^{k-2}\gamma_{k_{1}}\left\lbrace \sum\limits_{j_{1} = 2}^{k-k_{1}} \dfrac{1}{j_{1}!}\left[ B_{1}\partial_{\varepsilon}^{j_{1}}\left(y_{k-k_{1}-j_{1}}\left(t - \omega\left(\varepsilon\right)\right) \right)\big\vert_{\varepsilon=0}	\right.	\right.	\\
&\hspace{26mm}
\left.	\textcolor{white}{\sum\limits_{j_{1} = 2}^{k-k_{1}}} \left. + B_{2}\partial_{\varepsilon}^{j_{1}}\left(y_{k-k_{1}-j_{1}}\left(t - s_{0}\omega\left(\varepsilon\right)\right) \right)\big\vert_{\varepsilon=0}\right] \right\rbrace.
\end{split}
\end{equation}

Expanding $y_{j}(t)$ in Fourier series for $1\leq j\leq k$,
\begin{equation*}
y_{j}(t) = \sum\limits_{n\in\mathbb{Z}} \hat{y}_{j}(n)e^{i n t},
\end{equation*}
and defining,
\begin{equation}
\tilde{\Delta} (\lambda,\gamma,z)
=
\left(
\begin{array}{cc}
\lambda+\gamma\left(-\dfrac{2 u_{0}}{v_{0}}+b\right) & \gamma\dfrac{u_{0}^{2}}{v_{0}^{2}}e^{-z s_{0}} \\ \\
            -2\gamma u_{0}e^{-z}          &         \lambda+\gamma
\end{array}
\right),
\end{equation}
then we observe that expanding (\ref{linear_part_of_DDE_perturbed_at_order_epsilon_k}) in Fourier series, the $n$-th coefficient is,
\begin{equation*}
\Delta(n i\omega_{0}, \gamma_{0})\hat{y}_{k}(n)
+
\left[\tilde{\Delta}(n i\omega_{k-1}, \gamma_{k-1}, n i\omega_{0}) + \gamma_{0}n i\omega_{k-1}\left(e^{-n i\omega_{0}}B_{1} + s_{0}e^{-n i\omega_{0}s_{0}}B_{2}\right)\right]\hat{y}_{1}(n) + \hat{L}_{k}(n),
\end{equation*}
where, using the notation $e_{k}$ and $\tilde{e}_{k}$ of the \textbf{Appendix},
\begin{equation*}
\begin{split}
\hat{L}_{k}(n) =
&-
\gamma_{0}\bigl[\tilde{e}_{k-2}\left(\omega_{0},\dots,\omega_{k-2};n;1\right)B_{1} + \tilde{e}_{k-2}\left(\omega_{0},\dots,\omega_{k-2};n;s_{0}\right)B_{2}\bigr]\hat{y}_{1}(n)	\\
&+
\sum\limits_{k_{1}=2}^{k-2}\left\lbrace \tilde{\Delta}\left(n i\omega_{k_{1}}, \gamma_{k_{1}},n i\omega_{0}\right)
- \gamma_{0}\bigl[ e_{k_{1}}\left(\omega_{0},\dots,\omega_{k_{1}};n;1\right)B_{1} + e_{k_{1}}\left(\omega_{0},\dots,\omega_{k_{1}}; n; s_{0}\right)B_{2}\bigr]\right\rbrace \hat{y}_{k-k_{1}}(n) \\
&-
\sum\limits_{k_{1}=2}^{k-2} \gamma_{k_{1}}\left\lbrace \sum\limits_{j_{1}=2}^{k-k_{1}} \left[e_{j_{1}}\left(\omega_{0},\dots,\omega_{j_{1}}; n; 1\right)B_{1} + e_{j_{1}}\left(\omega_{0},\dots,\omega_{j_{1}}; n; s_{0}\right)B_{2}\right]\hat{y}_{k-k_{1}-j_{1}}(n)\right\rbrace.
\end{split}
\end{equation*}

We note that the expression of $\hat{L}_{k}(n)$ has been obtained using automatic differentiation (see \cite{haro_canadell_figueras_luque_mondelo}).

On the other hand, we observe that $N(y(t,0), y(t-\omega_{0},0), y(t-s_{0}\omega_{0},0)) = N(0,0,0) = 0$ and by Taylor's theorem,
\begin{equation*}
N(y(t,\varepsilon), y(t-\omega(\varepsilon),\varepsilon),y (t-s_{0}\omega(\varepsilon),\varepsilon))
=
\sum\limits_{k=0}^{\infty} \dfrac{1}{k!} \partial_{\varepsilon}^{k} \bigl( N(y(t,\varepsilon),y(t-\omega)(\varepsilon),\varepsilon), y(t-s_{0}\omega(\varepsilon),\varepsilon))\bigr)\big\vert_{\varepsilon=0}\varepsilon^{k}.
\end{equation*}

In order to obtain an explicit expression of $N(y(t,\varepsilon), y(t-\omega(\varepsilon),\varepsilon),y (t-s_{0}\omega(\varepsilon),\varepsilon))$ as a Taylor series, we observe that,
\begin{equation*}
\begin{split}
\partial_{\varepsilon}^{1}\left(N (y(t,\varepsilon),y (t-\omega(\varepsilon),\varepsilon),y(t-s_{0}\omega(\varepsilon),\varepsilon))\right)
&=
\begin{pmatrix}
2\left(\dfrac{U\left(t,\varepsilon\right) + u_{0}}{V\left(t-s_{0}\omega(\varepsilon)\right) + v_{0}}\right)\partial_{\varepsilon}\left(U\left(t,\varepsilon\right)\right) - \dfrac{2u_{0}}{v_{0}}\partial_{\varepsilon}\left(U\left(t,\varepsilon\right)\right)	\\	\\
0
\end{pmatrix}	\\
&\hspace{4mm}+
\begin{pmatrix}
0	\\	\\
2 U\left(t-\omega(\varepsilon), \varepsilon\right)\partial_{\varepsilon}\left(U\left(t-\omega(\varepsilon), \varepsilon\right)\right)
\end{pmatrix}	\\
&\hspace{4mm}+
\begin{pmatrix}
\dfrac{u_{0}^{2}}{v_{0}^{2}}\partial_{\varepsilon}\left(V\left(t-s_{0}\omega(\varepsilon),\varepsilon\right)\right)	\\	\\
0
\end{pmatrix}	\\
&\hspace{4mm}-
\begin{pmatrix}
\left(\dfrac{U\left(t,\varepsilon\right) + u_{0}}{V\left(t-s_{0}\omega(\varepsilon)\right) + v_{0}}\right)^{2}\partial_{\varepsilon}\left(V\left(t-s_{0}\omega(\varepsilon),\varepsilon\right)\right)	\\	\\
0
\end{pmatrix},
\end{split}
\end{equation*}
therefore $\partial_{\varepsilon}^{1}\left(N (y(t,\varepsilon),y (t-\omega(\varepsilon),\varepsilon),y(t-s_{0}\omega(\varepsilon),\varepsilon))\right)\vert_{\varepsilon=0}=0$.

The idea is expanding all expressions in Taylor series in $\varepsilon$. For this we observe that,
\begin{equation*}
\dfrac{U(t,\varepsilon) + u_{0}}{V(t-s_{0}\omega(\varepsilon),\varepsilon) + v_{0}}
=
\sum\limits_{k=0}^{\infty} d_{k}\varepsilon^{k},
\end{equation*}
where (see for example \cite{haro_canadell_figueras_luque_mondelo}),
\begin{center}
$d_{0}=\dfrac{u_{0}}{v_{0}}$,	\\
$d_{k}=\dfrac{1}{v_{0}} \left\lbrace U_{k}(t) - \sum\limits_{l=0}^{k-1}d_{l}\left[\sum\limits_{j_{1}=0}^{k-l} \dfrac{1}{j_{1}!}\partial_{\varepsilon}^{j_{1}}\left(V_{k-l-j_{1}}(t-s_{0}\omega(\varepsilon)\right)\big\vert_{\varepsilon=0}\right]\right\rbrace$ for $k\geq1$.
\end{center}

Then consider the following lemma.
\begin{lemma}
Suppose that for $k\geq 2$, we have
\begin{equation*}
\begin{split}
d_{k}
&=
d_{k}(t)	\\
&=
\sum\limits_{n\in\mathbb{Z}} \hat{d}_{k}(n)e^{i n t}, \hspace{2mm} \hat{d}_{k}(n) = \operatorname{conj}\left(\hat{d}_{k}(-n)\right)\in\mathbb{C}\text{ for all } n\in\mathbb{Z},
\end{split}
\end{equation*}
then
\begin{equation*}
\hat{d}_{k}(n)
=
\dfrac{d_{0}}{v_{0}}n i\omega_{k-1}s_{0}\hat{V}_{1}(n)e^{-n i\omega_{0}s_{0}}
+
h\left(\hat{U}_{1}, \hat{V}_{1},\dots, \hat{U}_{k}, \hat{V}_{k};\omega_{0},\dots, \omega_{k-2} \right),
\end{equation*}
where $h$ is an expression with terms that only depend of $\hat{U}_{1}, \hat{V}_{1},\dots, \hat{U}_{k}, \hat{V}_{k};\omega_{0},\dots, \omega_{k-2}$.
\end{lemma}
\begin{proof}
	Observe that,
	\begin{equation*}
		\begin{split}
			d_{1}
			&=
			d_{1}(t)	\\
			&=
			\dfrac{1}{v_{0}}\left[U_{1}(t) - d_{0}V_{1}(t - s_{0}\omega_{0}) \right]	\\
			&=
			\sum\limits_{n\in\mathbb{Z}}\hat{d}_{1}(n)e^{i n  t},
		\end{split}
	\end{equation*}
	where $\hat{d}_{1}(n)=\dfrac{1}{v_{0}}\left[\hat{U}_{1}(n) - d_{0}\hat{V}_{1}(n)e^{-n i\omega_{0}s_{0}} \right]$.
We proceed by induction on $k$.

The claim is true for $k=2$ and $k=3$ since,
\begin{equation*}
	\begin{split}
		d_{2}
		&=
		d_{2}(t) \\
		&=
		\sum\limits_{n\in\mathbb{Z}}\hat{d}_{2}(n) e^{i n t},
	\end{split}
\end{equation*}
where $\hat{d}_{2}(n)
=
\dfrac{d_{0}}{v_{0}}n i\omega_{1}s_{0}\hat{V}_{1}(n)e^{-n i\omega_{0}s_{0}}
+
\dfrac{1}{v_{0}}\left[\hat{U}_{2}(n) - d_{0}\hat{V}_{2}(n)e^{-n i\omega_{0}s_{0}} - \left(\hat{d}_{1}\ast\hat{V}_{1}(\cdot)e^{-(\cdot)i\omega_{0}s_{0}}\right)(n)\right]
$\footnote
{
where $\hat{g}_{1}*\hat{g}_{2}$ is the Cauchy product of $\hat{g}_{1}$ with $\hat{g}_{2}$, given by $\left(\hat{g}_{1}*\hat{g}_{2} \right)(n) = \sum\limits_{\substack{n_{1} + n_{2} = n\\ n_{1,2}\in\mathbb{Z}}} \hat{g}_{1}(n_{1})\hat{g}_{2}(n_{2})$ 
}
and,
\begin{equation*}
	\begin{split}
	d_{3}
	&=
	d_{3}(t)	\\
	&=
	\sum\limits_{n\in\mathbb{Z}}\hat{d}_{3}(n) e^{i n t},
	\end{split}
\end{equation*}
with,
\begin{equation*}
	\begin{split}
		\hat{d}_{3}(n)
		&=
		\dfrac{d_{0}}{v_{0}}n i\omega_{2}s_{0}\hat{V}_{1}(n)e^{-n i\omega_{0}s_{0}}	\\
		&\hspace{4mm}+
		\dfrac{1}{v_{0}}\left\lbrace\hat{U}_{3}(n) - d_{0}\left[\hat{V}_{3}(n)e^{-n i\omega_{0}s_{0}} - n i\omega_{1}s_{0}\hat{V}_{2}(n)e^{-n i\omega_{0}s_{0}} - \dfrac{1}{2}n^{2}\omega_{1}^{2}s_{0}^{2}\hat{V}_{1}(n)e^{-n i\omega_{0}s_{0}} \right]\right.	\\
		&\hspace{16mm}-
		\left[\hat{d}_{1}\ast\left(\hat{V}_{2}(\cdot)e^{-(\cdot) i\omega_{0}s_{0}} - (\cdot) i\omega_{1}s_{0}\hat{V}_{1}(\cdot)e^{-(\cdot)i\omega_{0}s_{0}}\right)\right](n)
		-
		\left[\hat{d}_{2}\ast\hat{V}_{1}(\cdot)e^{-(\cdot)i\omega_{0}s_{0}}\right](n) \bigg \}.
	\end{split}
\end{equation*}

Suppose that $\hat{d}_{2}(n)$, $\hat{d}_{3}(n)$, $\dots$, $\hat{d}_{k-1}(n)$ satisfy the property. Now, we show that $\hat{d}_{k}(n)$ satisfies the property. Observe that, now $k\geq 4$ and,
\begin{equation*}
	\begin{split}
		d_{k}
		&=
		d_{k}(t)	\\
		&=
		\sum\limits_{n\in\mathbb{Z}}\hat{d}_{k}(n) e^{i n t},
	\end{split}
\end{equation*}
where, using the induction hypotesis,
\begin{equation*}
	\begin{split}
		\hat{d}_{k}(n)
		&=
		\dfrac{d_{0}}{v_{0}}n i\omega_{k-1}s_{0}\hat{V}_{1}(n)e^{-n i\omega_{0}s_{0}}	\\
		&\hspace{4mm}+
		\dfrac{1}{v_{0}}\left\lbrace \hat{U}_{k}(n) - d_{0}\sum\limits_{j_{1}=0}^{k-2}e_{j_{1}}\left(\omega_{0},\cdots,\omega_{j_{1}};n;s_{0}\right)\hat{V}_{k-j_{1}}(n)
		-
		d_{0}\tilde{e}_{k-2}\left(\omega_{0},\cdots,\omega_{k-2};n;s_{0}\right)\hat{V}_{1}(n)	\right.	\\
		&\hspace{25mm}-
		\left[\hat{d}_{1}\left(\hat{U}_{1},\hat{V}_{1};\omega_{0};\cdot\right)\ast\left(\sum\limits_{j_{1}=0}^{k-2}e_{j_{1}}\left(\omega_{0},\dots,\omega_{j_{1}};\cdot;s_{0}\right)\hat{V}_{k-1-j_{1}}(\cdot)\right)\right](n)	\\
		&\hspace{25mm}-
		\sum\limits_{l=2}^{k-1}\left[\hat{d}_{l}\left(\hat{U}_{1},\hat{V}_{1},\dots,\hat{U}_{l},\hat{V}_{l};\omega_{0},\dots,\omega_{l-1};\cdot\right)\ast \textcolor{white}{\left[\sum\limits_{j_{1}=0}^{k}\right]}	\right.	\\
		&\hspace{40mm}\ast
		\left. \left. \left(\sum\limits_{j_{1}=0}^{k-1-l}e_{j_{1}}\left(\omega_{0},\dots,\omega_{j_{1}};\cdot;s_{0}\right)\hat{V}_{k-l-j_{1}}(\cdot)\right) \right](n)\right\rbrace.
	\end{split}
\end{equation*}
\end{proof}

Thus, we have that for $k\geq 1$,
\begin{equation*}
\begin{split}
d_{k}
&=
d_{k}(t) \\
&=
d_{k}(U_{1}, V_{1},\dots,U_{k}, V_{k}, \omega_{0},\dots,\omega_{k-1})(t).
\end{split}
\end{equation*}

In this form we have the Taylor coefficients for the following.

The term, $
\left(\dfrac{U(t,\varepsilon)+u_{0}}{V(t-s_{0}\omega(\varepsilon),\varepsilon)+v_{0}}\right)\partial_{\varepsilon}(U(t,\varepsilon))
$,
at order $\varepsilon^{k}$ is
$
\sum\limits_{k_{1}=0}^{k}(k-k_{1}+1) d_{k_{1}}(t) U_{k-k_{1}+1}(t)
$
for $k\geq 0$.

The term, $
2 \left(\dfrac{U(t,\varepsilon) + u_{0}}{V(t-s_{0} \omega(\varepsilon),\varepsilon) + v_{0}}\right) \partial_{\varepsilon} \left(U(t,\varepsilon)\right) - \dfrac{2u_{0}}{v_{0}} \partial_{\varepsilon} \left(U(t,\varepsilon)\right)
$,
at order $\varepsilon^{k}$ is $0$ for $k = 0$ and $2 \sum\limits_{k_{1}=1}^{k} (k-k_{1} + 1) d_{k_{1}}(t) U_{k-k_{1}+1}(t)$ for $k\geq 1$.

The term, $2U (t-\omega(\varepsilon),\varepsilon) \partial_{\varepsilon} (U(t-\omega(\varepsilon),\varepsilon))$, at order $\varepsilon^{k}$ is $0$ for $k = 0$ and \\
 $2 \sum\limits_{k_{1}=1}^{k} (k-k_{1}+1) \left[\sum\limits_{j_{1}=0}^{k_{1}} \dfrac{1}{j_{1}!} \partial_{\varepsilon}^{j_{1}} (U_{k_{1}-j_{1}} (t-\omega(\varepsilon)))\vert_{\varepsilon=0}\right] \left[\sum\limits_{j_{1}=0}^{k-k_{1}+1} \dfrac{1}{j_{1}!} \partial_{\varepsilon}^{j_{1}} (U_{k-k_{1}+1-j_{1}} (t-\omega(\varepsilon)))\vert_{\varepsilon=0}\right]$ for $k\geq1$.

The term, 
$
\dfrac{u_{0}^{2}}{v_{0}^{2}} \partial_{\varepsilon} (V(t-s_{0}\omega(\varepsilon),\varepsilon))-\left(\dfrac{U(t,\varepsilon)+u_{0}}{V(t-s_{0}\omega(\varepsilon),\varepsilon)+v_{0}}\right)^{2} \partial_{\varepsilon} (V(t-s_{0}\omega(\varepsilon),\varepsilon))
$,
at order $\varepsilon^{k}$ is $0$ for $k = 0$ and\\
$-\sum\limits_{k_{1}=1}^{k} (k-k_{1}+1) \left[\sum\limits_{j_{1=0}}^{k_{1}} d_{j_{1}}(t) d_{k_{1}-j_{1}}(t)\right]\left[\sum\limits_{j_{1}=0}^{k-k_{1}+1} \dfrac{1}{j_{1}!} \partial_{\varepsilon}^{j_{1}} (V_{k-k_{1}+1-j_{1}} (t-s_{0}\omega(\varepsilon)))\vert_{\varepsilon=0} \right]$ for $k\geq1$.

Hence, defining $\tilde{N}_{0}^{(1)}(t) = \tilde{N}_{0}^{(2)}(t) = 0$ and for $k\geq1$,
\begin{equation}\label{N_k_1_tilde}
\begin{split}
\tilde{N}_{k}^{(1)} (t)
&=
\sum\limits_{k_{1}=1}^{k} (k-k_{1}+1)
\left\lbrace
2
d_{k_{1}} (t) U_{k-k_{1}+1}(t)\textcolor{white}{\left[\sum\limits_{j_{1}=0}^{k}\right]}\right.	\\
&\hspace{32mm}-
\left. \left[ \sum\limits_{j_{1}=0}^{k_{1}} d_{j_{1}} (t) d_{k_{1}-j_{1}} (t) \right]  
\left[  \sum\limits_{j_{1}=0}^{k-k_{1}+1} \dfrac{1}{j_{1}!} \partial_{\varepsilon}^{j_{1}} (V_{k-k_{1}+1-j_{1}} (t-s_{0}\omega(\varepsilon))) \vert_{\varepsilon=0} \right] \right\rbrace,
\end{split}
\end{equation}
and,
\begin{equation}\label{N_k_2_tilde}
	\begin{split}
		\tilde{N}_{k}^{(2)} (t)
		&=
		2 \sum\limits_{k_{1}=1}^{k} (k-k_{1}+1) \left[\sum\limits_{j_{1}=0}^{k_{1}} \dfrac{1}{j_{1}!} \partial_{\varepsilon}^{j_{1}} (U_{k_{1}-j_{1}}(t-\omega(\varepsilon))) \vert_{\varepsilon=0}\right]\times	\\
		&\hspace{32mm}\times
 		\left[\sum\limits_{j_{1}=0}^{k-k_{1}+1} \dfrac{1}{j_{1}!} \partial_{\varepsilon}^{j_{1}} (U_{k-k_{1}+1-j_{1}} (t-\omega(\varepsilon))) \vert_{\varepsilon=0}\right],
	\end{split}
\end{equation}
then the functions $2\pi$-periodic $\tilde{N}_{k} (t)
=
\left(
\tilde{N}_{k}^{(1)} (t),
\tilde{N}_{k}^{(2)} (t)
\right)^{T}
$
for $k\geq 0$,
satisfies,
\begin{equation*}
\partial_{\varepsilon}(N(y(t,\varepsilon), y(t-\omega(\varepsilon),\varepsilon),y (t-s_{0}\omega(\varepsilon),\varepsilon))) = \sum\limits_{k=0}^{\infty} \tilde{N}_{k} (t)\varepsilon^{k}.
\end{equation*}

We observe that in (\ref{N_k_1_tilde}), the term $d_{j}(t)$ its write as $d_{j}(t) = d_{j}(U_{1}, V_{1},\dots, U_{j}, V_{j}; \omega_{0},\dots,\omega_{j-1})(t)$, hence when $j=k$ then $d_{k}$ contains $U_{k}, V_{k}$ and $\omega_{k-1}$. For the term $\partial_{\varepsilon}^{j_{1}}\left( V_{k-k_{1} + 1 - j_{1}} \left( t - s_{0}\omega\left( \varepsilon \right) \right) \right)\vert_{\varepsilon=0}$, given that $0\leq j_{1} \leq k-k_{1}+1$ and $k-k_{1}+1\leq k$, then with $k_{1}=1$ we have $\sum\limits_{j_{1}=1}^{k}\dfrac{1}{j_{1}!}\partial_{\varepsilon}^{j_{1}}\left( V_{k-j_{1}}(t - s_{0}\omega(\varepsilon)) \right)\vert_{\varepsilon=0}$ and in this sum we obtain $\omega_{k-1}$ when $j_{1}=k-1$. When $k\geq 2$ then we obtain $\omega_{j}$ with $j<k-1$. Therefore the function $\tilde{N}_{k}(t)$ depends only on $U_{1}, V_{1},\dots, U_{k}, V_{k}$ and $\omega_{0},\dots,\omega_{k-1}$ for $k\geq 1$.

In this way, 
\begin{equation*}
\partial_{\varepsilon}^{k}\left(N (y(t,\varepsilon),y (t-\omega(\varepsilon),\varepsilon),y(t-s_{0}\omega(\varepsilon),\varepsilon))\right)\vert_{\varepsilon=0}=(k-1)!\tilde{N}_{k-1}(U_{1}, V_{1},\dots, U_{k-1}, V_{k-1};\omega_{0},\dots,\omega_{k-2}) (t)\varepsilon^{k}.
\end{equation*}

Thus, defining $N_{0} = N_{1} = 0$ and $N_{k} = \dfrac{1}{k} \tilde{N}_{k-1} (U_{1}, V_{1},\dots, U_{k-1}, V_{k-1};\omega_{0},\ldots,\omega_{k-2}) (t)$ for $k\geq 2$, we obtain,
\begin{equation*}
\begin{split} 
\gamma (\varepsilon) N(y(t,\varepsilon),y (t-\omega(\varepsilon),\varepsilon),y(t-s_{0}\omega(\varepsilon),\varepsilon))
&=
\sum\limits_{k=0}^{\infty} \left(  \sum\limits_{k_{2}=0}^{k} \gamma_{k-k_{2}} N_{k_{2}} \right)\varepsilon^{k}	\\
&=
\sum\limits_{k=2}^{\infty} \left(  \sum\limits_{k_{2}=2}^{k} \gamma_{k-k_{2}} N_{k_{2}} \right)\varepsilon^{k}.
\end{split}
\end{equation*}

Thus, the right side of (\ref{DDE_GM_en_formato_N_equal_f_minus_L_delays_w_s0w_perturbed}) at order $\varepsilon^{k}$ is given by,
\begin{equation*}
	\begin{split}
		\sum\limits_{k_{2}=2}^{k} \gamma_{k-k_{2}} N_{k_{2}}
		&=
		\sum\limits_{n\in\mathbb{Z}}\left( \sum\limits_{k_{2}=2}^{k} \gamma_{k-k_{2}} \hat{N}_{k_{2}}(n) \right) e^{int}	\\
		&=
		\sum\limits_{n\in\mathbb{Z}} \hat{\rho}_{k} (n) e^{int},
	\end{split}
\end{equation*}
where $\hat{\rho}_{k} (n) = \sum\limits_{k_{2}=2}^{k} \gamma_{k-k_{2}} \hat{N}_{k_{2}} (n)$, with $\hat{N}_{k_{2}}(n)$ is the $n$-th Fourier coefficient of $N_{k_{2}}$. Therefore, (\ref{DDE_GM_en_formato_N_equal_f_minus_L_delays_w_s0w_perturbed}) is equivalent to solving,
\begin{equation}\label{lineal_Fourier_system_at_order_epsilon_power_k}
	\begin{split}
		\Delta (ni\omega_{0},\gamma_{0}) \hat{y}_{k}(n)
		&+
		\left[  \tilde{\Delta} (ni\omega_{k-1},\gamma_{k-1},ni\omega_{0}) + \gamma_{0}ni\omega_{k-1} (e^{-ni\omega_{0}} B_{1} + s_{0}e^{-ni\omega_{0}s_{0}} B_{2})  \right] \hat{y}_{1}(n)	\\
		&+
		\hat{L}_{k}(n)	\\
		&=
		\hat{\rho}_{k}(n).
	\end{split}
\end{equation}

Notice that $\left[ \tilde{\Delta} (ni\omega_{k-1},\gamma_{k-1},ni\omega_{0}) + \gamma_{0}ni\omega_{k-1} (e^{-ni\omega_{0}} B_{1} + s_{0} e^{-ni\omega_{0}s_{0}} B_{2}) \right] $ is the only term that contains $\gamma_{k-1}$ and $\omega_{k-1}$ and $\Delta (ni\omega_{0},\gamma_{0}) \hat{y}_{k}(n)$ is the only term that contains $\hat{y}_{k}(n)$. Defining $\hat{G}_{k}(n) = \hat{\rho}_{k}(n)-\hat{L}_{k}(n)$ and $\hat{R}_{k}(n)$ given by,
\begin{equation*}
\hat{R}_{k}(n)
=
-
\left[ \tilde{\Delta} (ni\omega_{k-1},\gamma_{k-1},ni\omega_{0}) + \gamma_{0}ni\omega_{k-1} (e^{-ni\omega_{0}} B_{1} + s_{0} e^{-ni\omega_{0}s_{0}} B_{2}) \right] \hat{y}_{1}(n)
+
\hat{G}_{k}(n),
\end{equation*}
then (\ref{lineal_Fourier_system_at_order_epsilon_power_k}) is equivalent to $\Delta (ni\omega_{0},\gamma_{0}) \hat{y}_{k}(n)=\hat{R}_{k}(n)$. Using Proposition \ref{proposition_for_finding_periodic_solutions} we have that $\Delta(i\omega_{0},\gamma_{0}) \hat{y}_{k}(1) = \hat{R}_{k}(1) $ has at least a solution if and only if,
\begin{equation*}
\begin{pmatrix}
	\gamma_{k-1} \\ \\
	\omega_{k-1}
\end{pmatrix} 
=
C^{-1}
\begin{pmatrix}
	\operatorname{Re}\left(\hat{\psi}(-1)\hat{G}_{k}(1) \right) \vspace{1mm} \\ 
	\operatorname{Im}\left(\hat{\psi}(-1)\hat{G}_{k}(1) \right)
\end{pmatrix}.
\end{equation*}

The last observation completes the proof of the existence of the Poincar\'e-Lindstedt series to all orders.
\end{proof}

\begin{remark}\label{remark_implementation_coefficients_of_Poincare-Lindstedt_series}
The existence of Poincar\'e - Lindstedt series to any order was proved. The expressions of the coefficients are explicit and recursive, meaning that at every order we can compute the next term of the series using previously computed terms. This makes it easy to numerical implement and is not necessary symbolic computation, optimizing computing resources.

We emphasize that the use of automatic differentiation provides formulas that are readily implementable. From the numerical point of view, this is a practical method whose implementation is simple since it only requires evaluating formulas with no need of symbolic manipulations.

The method presented above can be extended to nonlinear functions with asymptotic expansion and we plan to extend the study of the system of the form (\ref{DDE_GM_noncentered_delays_1_s0}) with diffusion for a future study. 

The easy implementation of Poincar'{e} - Lindstedt series allows it to be used as an initial guess for a correction method, such as a Newton method, which is the initial an principal use in our numerical simulations. We implemented the coefficients at any order, but for the Newton method the coefficients at order 3 is enough.
\end{remark}

\section{Numerical Results}
In this section we illustrate some numerical simulations of (\ref{DDE_GM_around_0}). Thus, lets consider the parameters $a$, $b$, $c$, and an equilibrium $\left(u_{0}, v_{0}\right)$ of (\ref{DDE_GM_noncentered_delays_1_s0}), where $\left(u_{0}, v_{0}\right)$ is such that the parameters in (\ref{eq:b0_b1_b2_tau_values}) satisfy the hypothesis of Proposition \ref{Hopf_bifurcation_theorem}. Three sets of parameters are analyzed: $a = \frac{1}{10}$, $b = \frac{11}{60}$, $c = 11$ (Parameter Set 1) with the corresponding equilibrium $\left(u_{0}, v_{0}\right)=(3,20)$; $a = \frac{1}{10}$ ,$b = 1$, $c = 10^{-6}$ (Parameter Set 2) with the corresponding equilibrium $\left(u_{0}, v_{0}\right) \approx (1.09999917, 1.20999918)$, and $a = \frac{1}{10}$, $b = 1$, $c = 0$ (Parameter Set 3) with the corresponding equilibrium $\left(u_{0}, v_{0}\right)=\left(1.1, 1.21 \right)$ (these parameter values coincide with the ones in \cite{murrayII}). For each parameter set we consider the delay values $s_{0}\in\left\lbrace 1.5, 2, 3, 5, 10 \right\rbrace$. We propose the idea of implement the values of delay $s_{0}$ and observe the behaviour of model (\ref{DDE_GM_around_0}).

Using Proposition \ref{Hopf_bifurcation_theorem} and remembering that $\gamma_{0}$ denotes the bifurcation parameter value, we compute the corresponding Hopf bifurcation points of system (\ref{DDE_GM_around_0}), which are shown in \textbf{Figure \ref{f:Initial_points_of_Hopf_bifurcation}}. For the Parameter Set 1 we consider $\gamma_{0}\in\left(0, 30 \right)$, and for the Parameter Set 2 and 3 we consider $\gamma_{0}\in\left(0, 2 \right)$ to obtain up to four Hopf bifurcation points. Observe that the last two parameter sets are very close and so are the graphs of the Hopf points.
\begin{figure}[h!]
 \centering
 \subfloat[Hopf bifurcation points using Parameter Set 1.]
 {
    \includegraphics[width=0.4\textwidth]{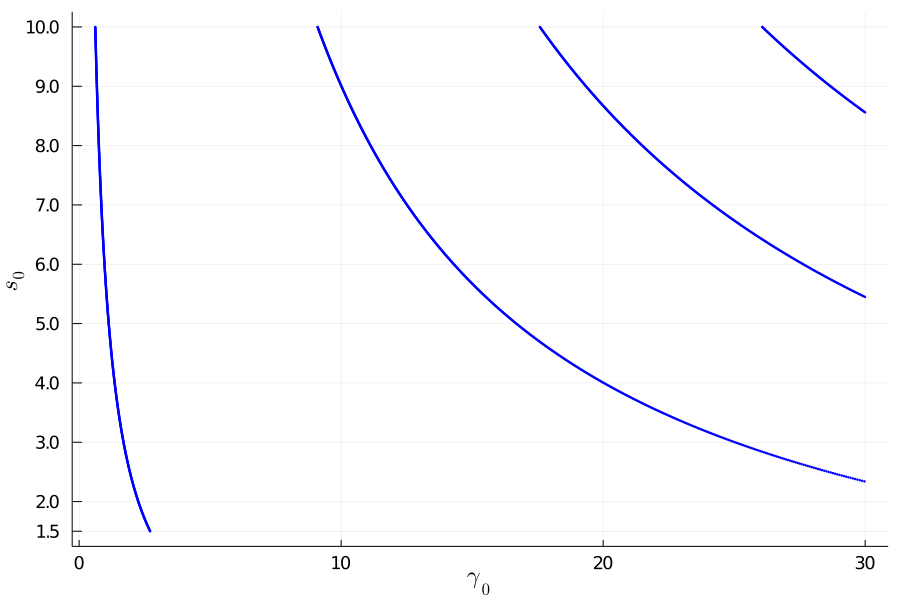}
 }
 \hspace{2mm}
  \subfloat[Hopf bifurcation points for Parameter Set 2 (\textcolor{green}{$\bullet$}) and Parameter Set 3 (\textcolor{red}{$\times$}).]
 {
   \label{f:figure_initial_points_for_Hopf_bifurcation_of_Par_Set_2_and_3}
    \includegraphics[width=0.4\textwidth]{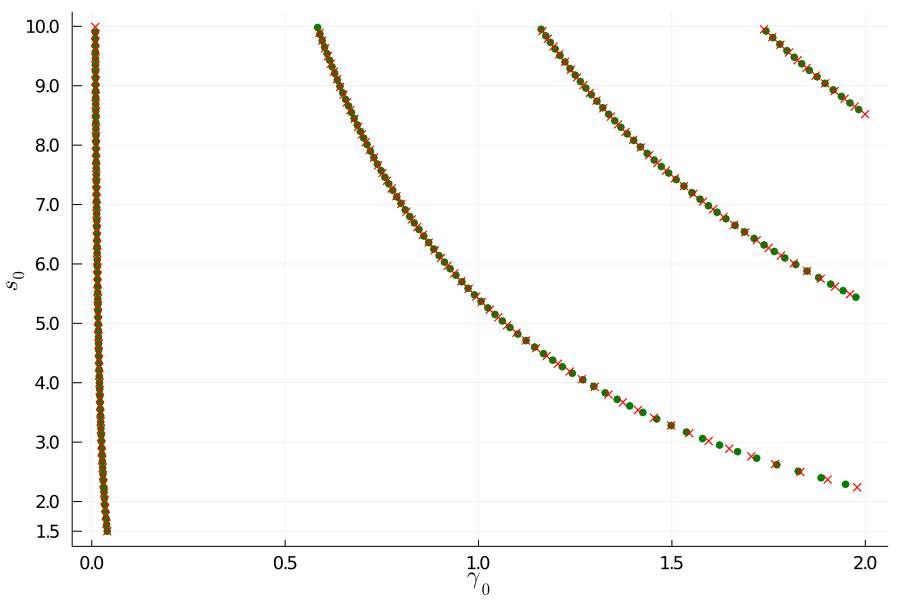}
 }
 \caption{Hopf bifurcation points for different Parameter Sets.}
 \label{f:Initial_points_of_Hopf_bifurcation}
\end{figure}

We implemented a Julia script to obtain coefficients of the Poincar\'e-Lindstedt series at any order of $y$, $\gamma$ and the frecuency $\omega$, given in (\ref{y_U_V_perturbed}) and (\ref{gamma_and_omega_perturbed}). We observe in \textbf{Figure \ref{f:figure_eps^k_by_yk}} that in our examples, the values of $\varepsilon^{k}\vert y_{k}\vert_{\infty}$ decreace numerically with particular values of $\varepsilon$. Note that the values of $\varepsilon$ of pannels \textbf{(\ref{f:figure_PS1_and_s0_eq_10})} and \textbf{(\ref{f:figure_PS2_and_s0_eq_10})} are comparable to pannels \textbf{(\ref{f:figure_PS1_and_s0_eq_1p5})} and \textbf{(\ref{f:figure_PS2_and_s0_eq_1p5})} up to a factor $10^{-1}$.
\begin{figure}[h!]
 \centering
 \subfloat[Parameter Set 1 and $s_{0}=1.5$]
 {
 	\label{f:figure_PS1_and_s0_eq_1p5}
    \includegraphics[width=0.4\textwidth]{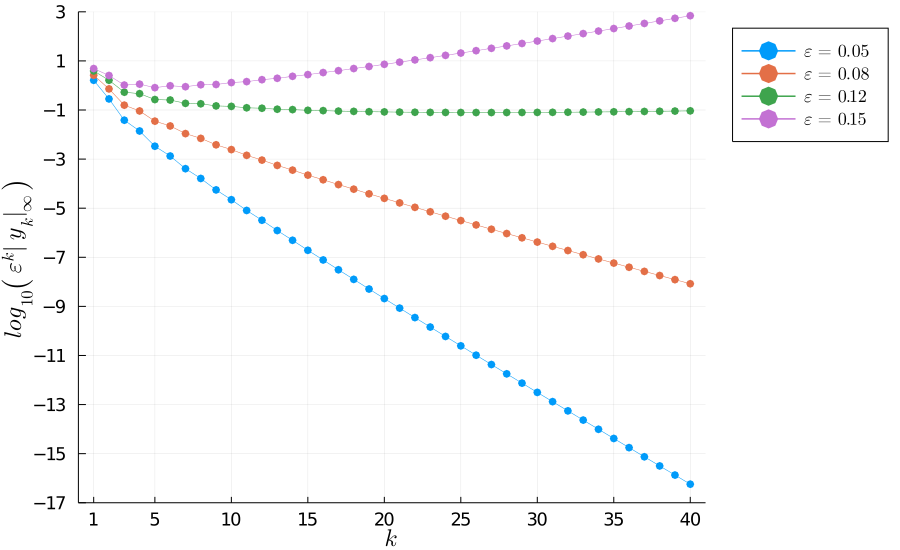}
 }
 \hspace{2mm}
  \subfloat[Parameter Set 1 and $s_{0}=10$]
 {
 	\label{f:figure_PS1_and_s0_eq_10}
    \includegraphics[width=0.4\textwidth]{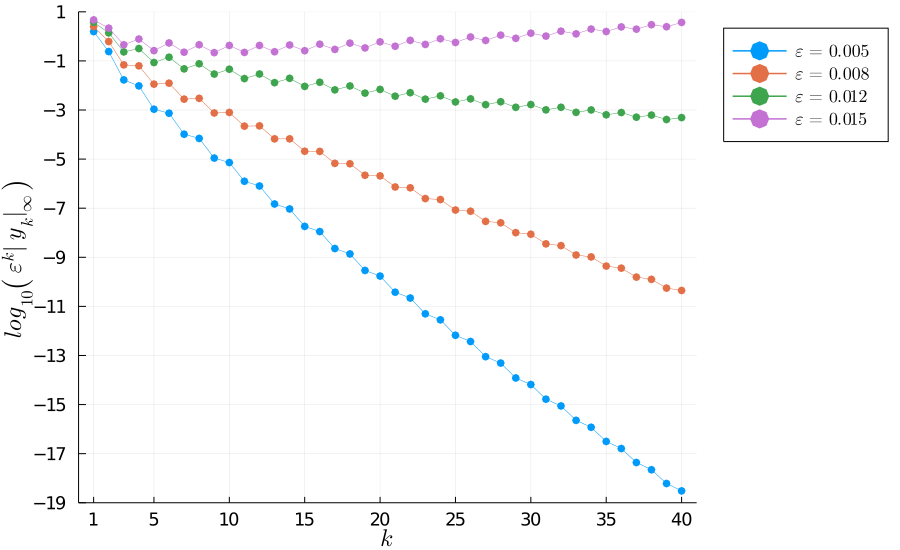}
 }
 \hspace{2mm}
  \subfloat[Parameter Set 2 and $s_{0}=1.5$]
 {
 	\label{f:figure_PS2_and_s0_eq_1p5}
    \includegraphics[width=0.4\textwidth]{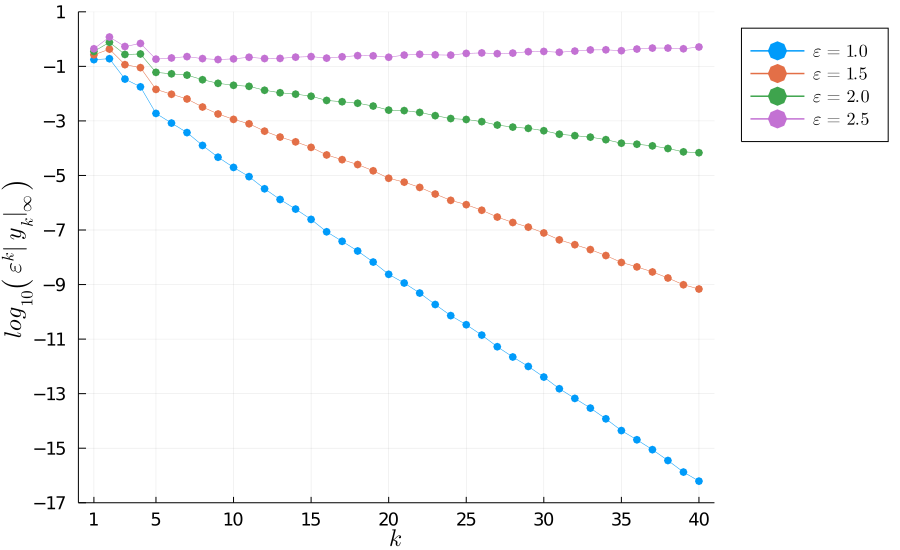}
 }
 \hspace{2mm}
  \subfloat[Parameter Set 2 and $s_{0}=10$]
 {
 	\label{f:figure_PS2_and_s0_eq_10}
    \includegraphics[width=0.4\textwidth]{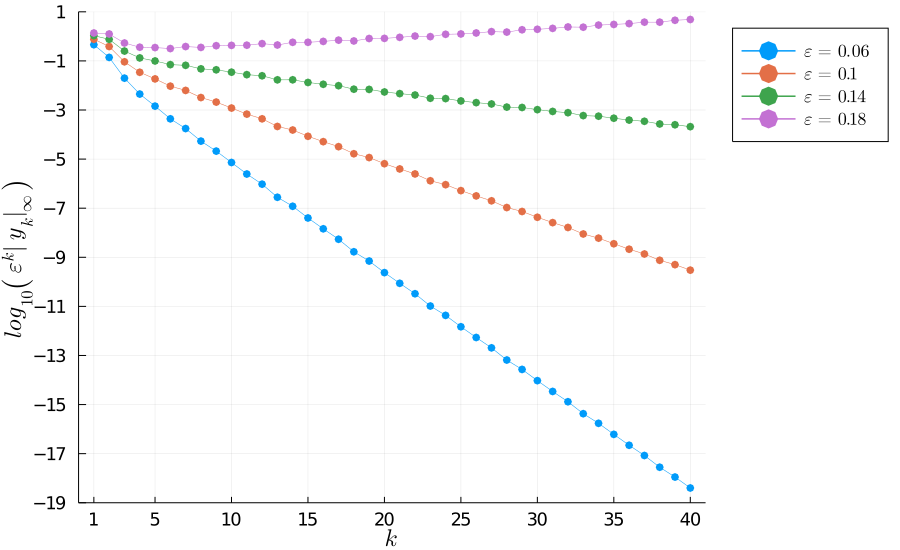}
 }
 \caption{Different Poincar\'e-Lindstedt series of periodic solutions $y$ of (\ref{DDE_GM_around_0}) and its behavior with particular values of $\varepsilon$.
}
 \label{f:figure_eps^k_by_yk}
\end{figure}

The main application of the Poincar\'e-Lindstedt series in this work is to approximate periodic solutions of (\ref{DDE_GM_around_0}) which are used as an initial guess to be corrected by a Newton method algorithm. For this, we use a collocation method (see for example \cite{doedel_keller_kernevezI}, \cite{doedel_keller_kernevezII}, \cite{engelborghs_luzyanina_hout_roose_01} and \cite{verheyden_lust_05}) over the following boundary value problem, transforming (\ref{DDE_GM_en formato_f}) in a delay differential equation with $1-$periodic solution,
\begin{subequations}\label{eq:bvp} 
	\begin{align} 
		T \gamma f(x(t), x(t-1/T), x(t-s_{0}/T))-x'(t)
		&=
		0 \text{ for }t\in\left[0,1 \right], \label{eq:eq1}\\
		x(\theta + 1) - x(\theta)
		&=
		0 \text{ for }\theta\in\left[-s_{0}/T,0 \right], \label{eq:periodicity_condition}\\
		\alpha(x, T)
		&=
		0, \label{eq:phase_conditon}
	\end{align} 
\end{subequations} 
where $\gamma$ is a knwon fixed value and $T$ denotes the (unknown) period. Equation (\ref{eq:periodicity_condition}) is the periodicity condition and (\ref{eq:phase_conditon}) represents a suitable phase condition to remove traslational invariance. The collocation method generates a Newton method that approximates points on the curve. The initial guess is given for the Poincar\'e-Lindstedt series at order $\varepsilon^{3}$ and this order is sufficient for the convergence of the Newton method in all our simulations. In \textbf{Figure \ref{f:figure_Poincare_Lindstedt_vs_Newton_method}} we comparate the distance beetwen $y_{PL}$ (the Poincar\'e-Lindstedt series at order $\varepsilon^{3}$) and $y_{NM}$ (the approximated periodic solution given by the Newton method). This refinement in the solution check again the usefulness of this perturbative method for approximating periodic solutions.
\begin{figure}[h!]
 \centering
 \subfloat[Distance between the points that lies on $y_{PL}$ at order $\varepsilon^{3}$ and $y_{NM}$ for the first bifurcation point of $s_{0} = 1.5$. This graph is identified as the blue point on $s_{0}=1.5$ of the right panel.]
 {
    \includegraphics[width=0.45\textwidth]{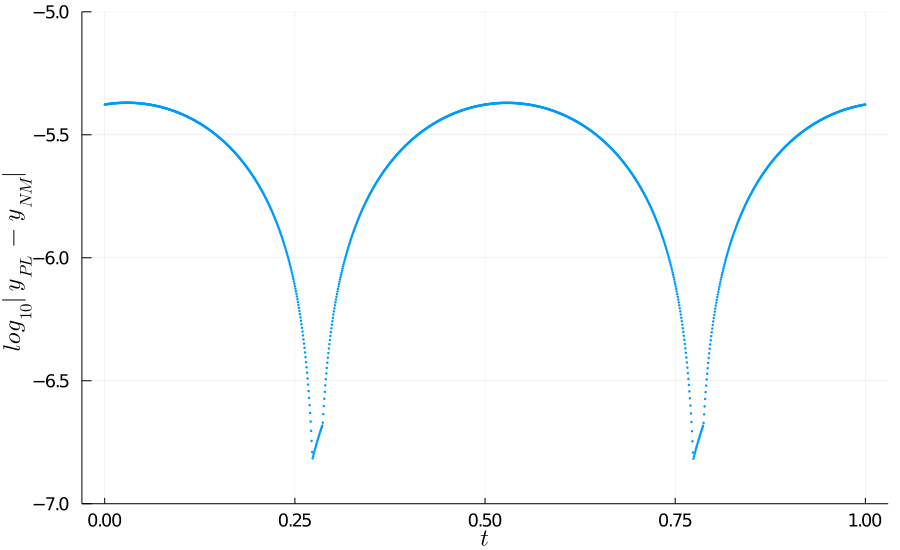}
 }
 \hspace{2mm}
  \subfloat[Distance between $y_{PL}$ at order $\varepsilon^{3}$ and the solution $y_{NM}$ for all simulations. The color \textcolor{blue}{blue} is for the Parameter Set 1, color \textcolor{red}{red} is for Parameter Set 2 and color \textcolor{green}{green} is for Parameter Set 3. Marker $\circ$ is the first branch, $\times$ is the second branch, $\diamond$ is the third branch and $\bigtriangleup$ is the fourth branch.]
 {
    \includegraphics[width=0.45\textwidth]{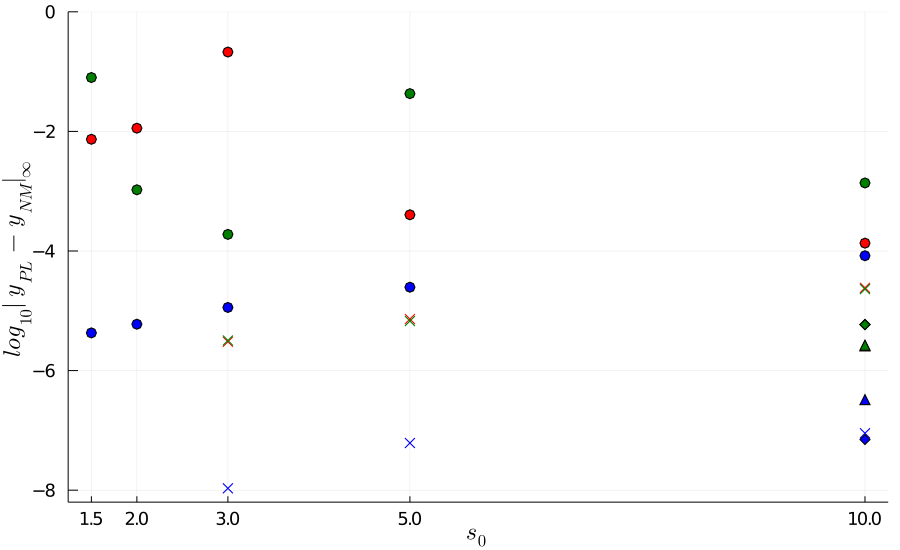}
 }
 \caption{Distance between the Poincar\'{e}-Lindstedt at order $\varepsilon^{3}$ and the solution obtained by a Newton method.}
 \label{f:figure_Poincare_Lindstedt_vs_Newton_method}
\end{figure}

Thus, a pseudo-arclength method also was implemented (see \cite{keller87}, \cite{doedel_keller_kernevezI} and \cite{lessard2018} for detailed examples of computation of branches of periodic solutions). The first branches of periodic solutions of model (\ref{DDE_GM_around_0}) for the Parameter Set 1 are shown in \textbf{Figure \ref{f:L2_principal_brances}}.
\begin{figure}[h!]
 \centering
 \subfloat[$s_{0} = 1.5$]
 {
    \includegraphics[width=0.35\textwidth]{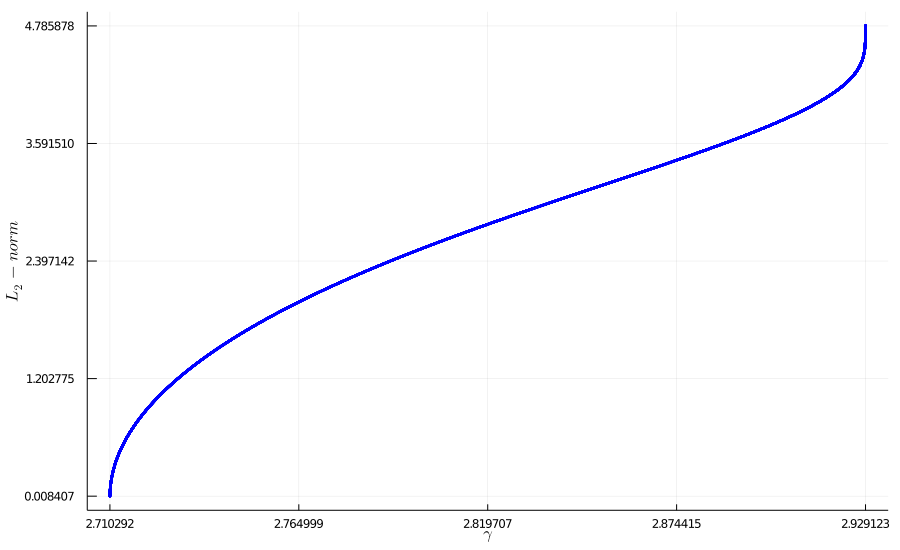}
 }
 \hspace{2mm}
  \subfloat[$s_{0} = 2$]
 {
    \includegraphics[width=0.35\textwidth]{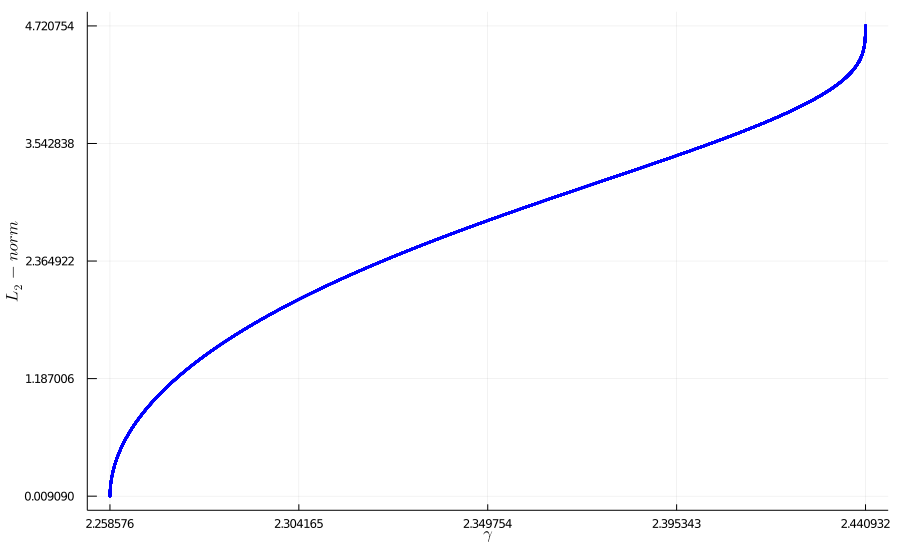}
 }
 \hspace{2mm}
  \subfloat[$s_{0} = 3$]
 {
    \includegraphics[width=0.35\textwidth]{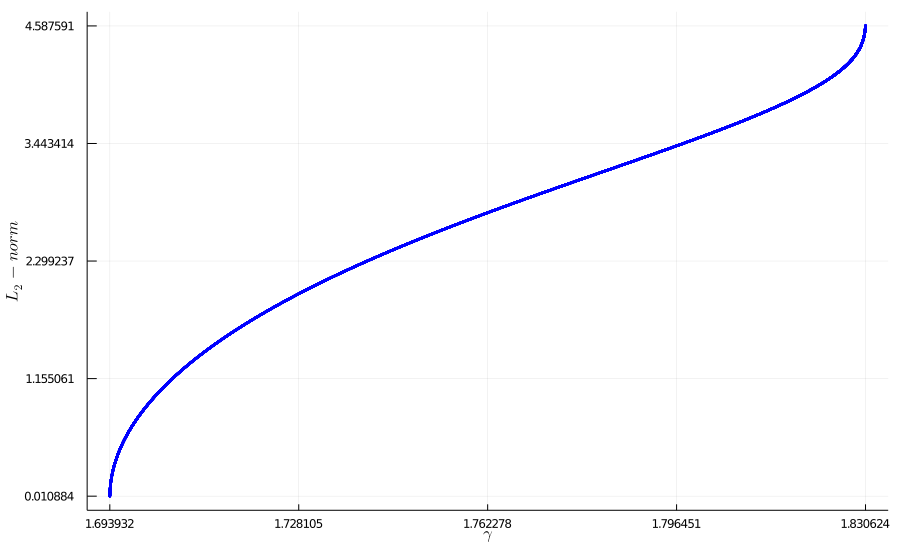}
 }
 \hspace{2mm}
  \subfloat[$s_{0} = 5$]
 {
    \includegraphics[width=0.35\textwidth]{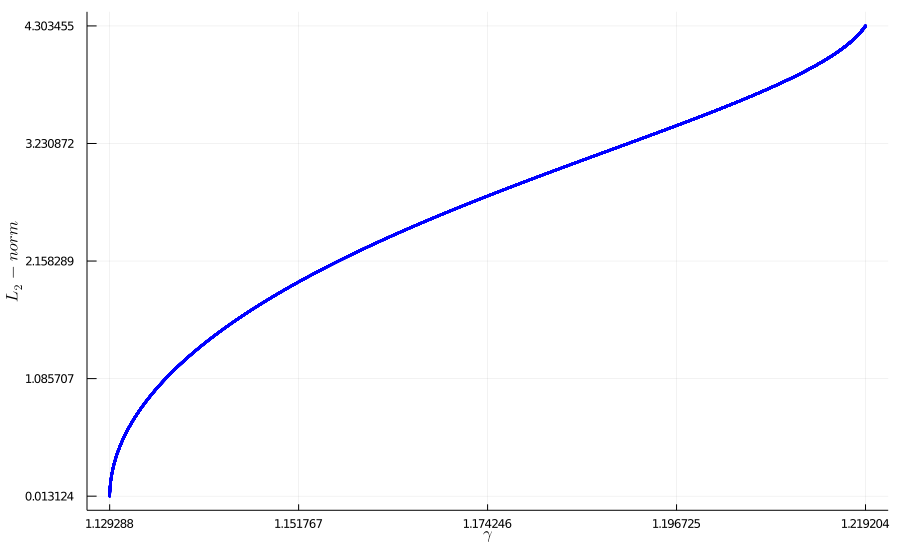}
 }
 \hspace{2mm}
  \subfloat[$s_{0} = 10$]
 {
    \includegraphics[width=0.35\textwidth]{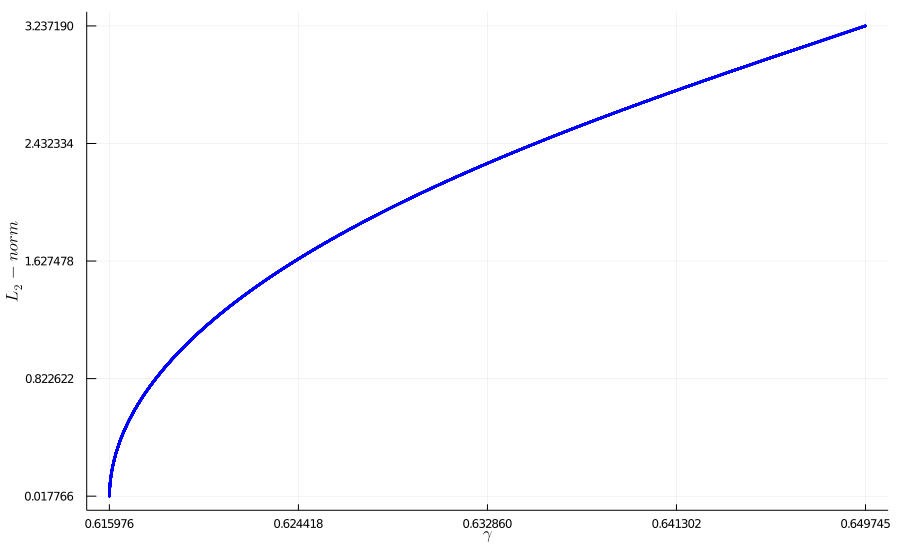}
 }
 \caption{First branches parameterized by $\gamma$ for the Parameter Set 1. Note the asymptotic behavior in almost all of these branches with the $L_{2}-$norm. These graphs are related with the Figure \ref{f:most_right_periodic_solutions_Par_set_1}, which suggests the presence of homoclinic orbits.}
 \label{f:L2_principal_brances}
\end{figure}

In \textbf{Figure \ref{f:figure_g_vs_NormInfty}} observe that in our simulations, all first continuation branches to bifurcate are close to each other. This situation occurs as we change the values of $s_{0}$. We notice that the bifurcation brances that emerge for large $\gamma_{0}$ values are also close to each other in all models that we consdered.
\begin{figure}[h!]
 \centering
 \subfloat[Parameter Set 1.]
 {
    \includegraphics[width=0.45\textwidth]{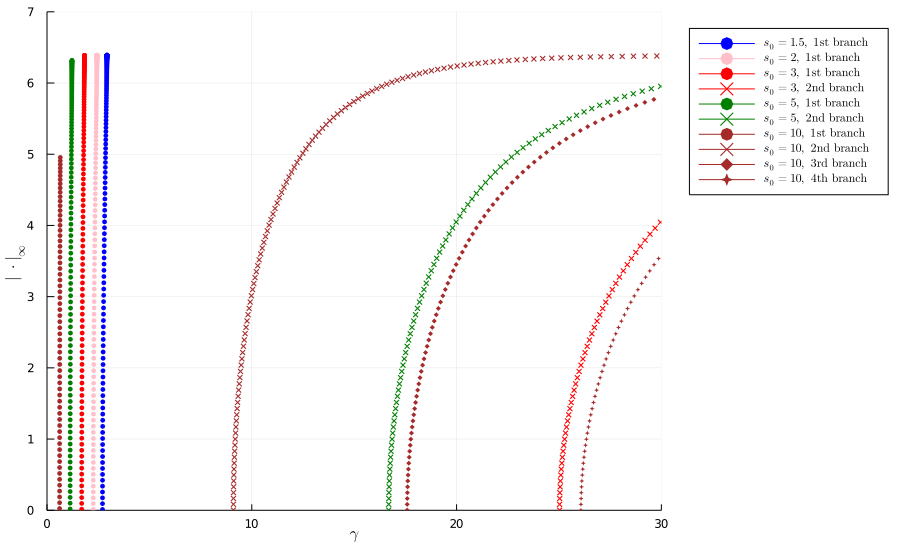}
 }
 \hspace{2mm}
  \subfloat[Parameter Set 3.]
 {
    \includegraphics[width=0.45\textwidth]{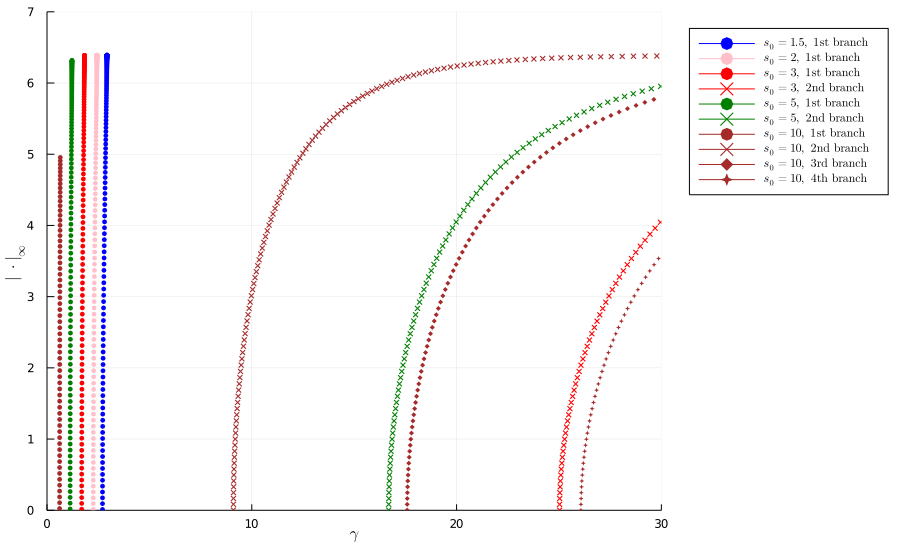}
 }
 \caption{Continuation of periodic solution of equation (\ref{DDE_GM_around_0}) with different parameter sets.}
 \label{f:figure_g_vs_NormInfty}
\end{figure}

In \textbf{Figure \ref{f:most_right_periodic_solutions_Par_set_1}} show the last computed periodic solutions of the continuation branches of \textbf{Figure \ref{f:L2_principal_brances}}. Note that the periodic solutions of \textbf{Figure \ref{f:most_right_periodic_solutions_Par_set_1}} are solutions of model (\ref{DDE_GM_noncentered_delays_1_s0}) after translating. 
\begin{figure}[h!]
 \centering
 \subfloat[$s_{0} = 1.5$, principal branch.]
 {
    \includegraphics[width=0.31\textwidth]{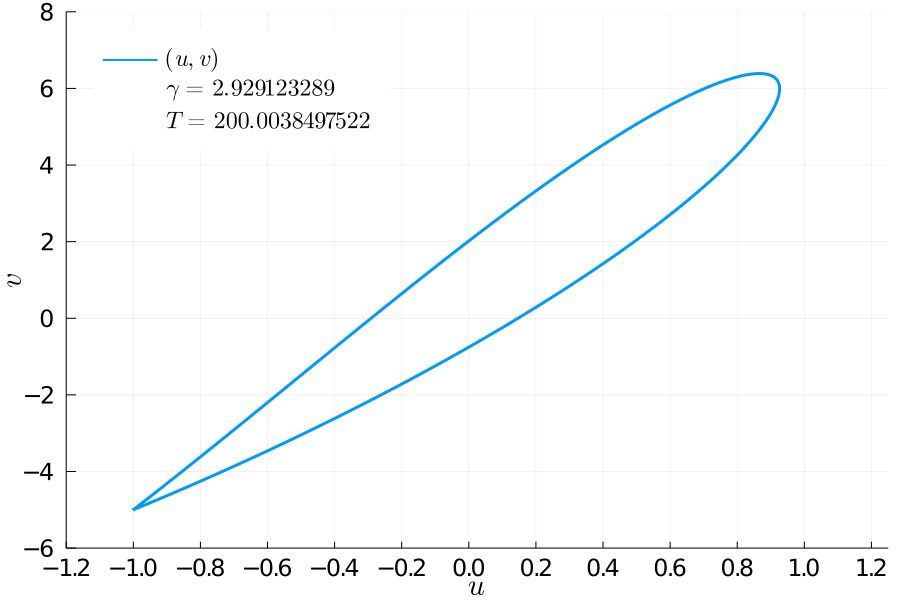}
 }
 \hspace{2mm}
  \subfloat[$s_{0} = 2$, principal branch.]
 {
    \includegraphics[width=0.31\textwidth]{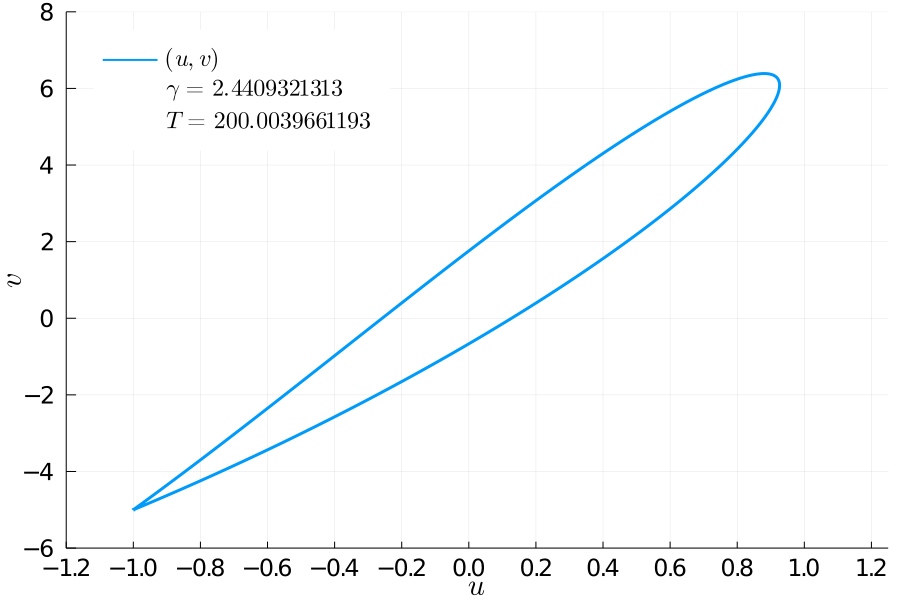}
 }
 \hspace{2mm}
  \subfloat[$s_{0} = 3$, principal branch.]
 {
    \includegraphics[width=0.31\textwidth]{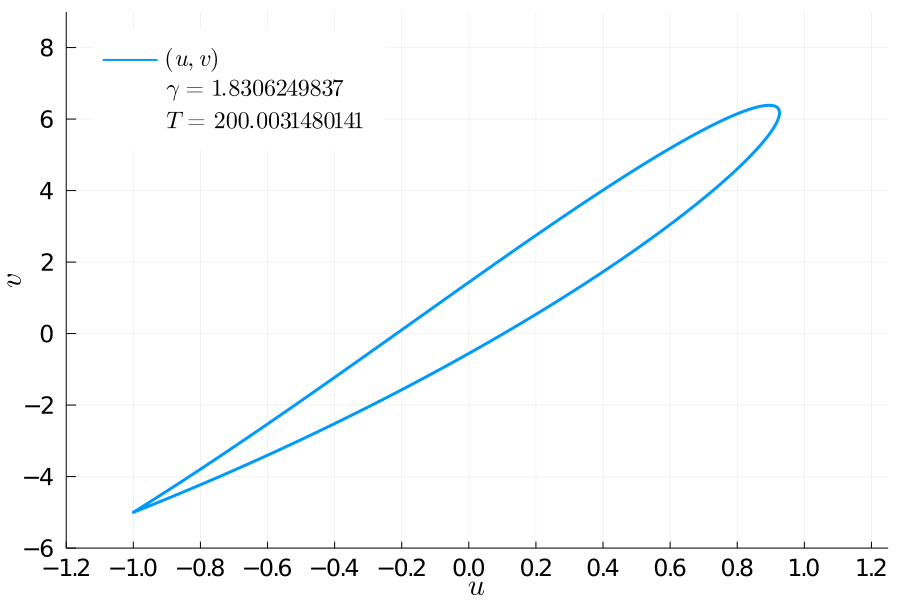}
 }
 \hspace{2mm}
  \subfloat[$s_{0} = 3$, second branch.]
 {
    \includegraphics[width=0.31\textwidth]{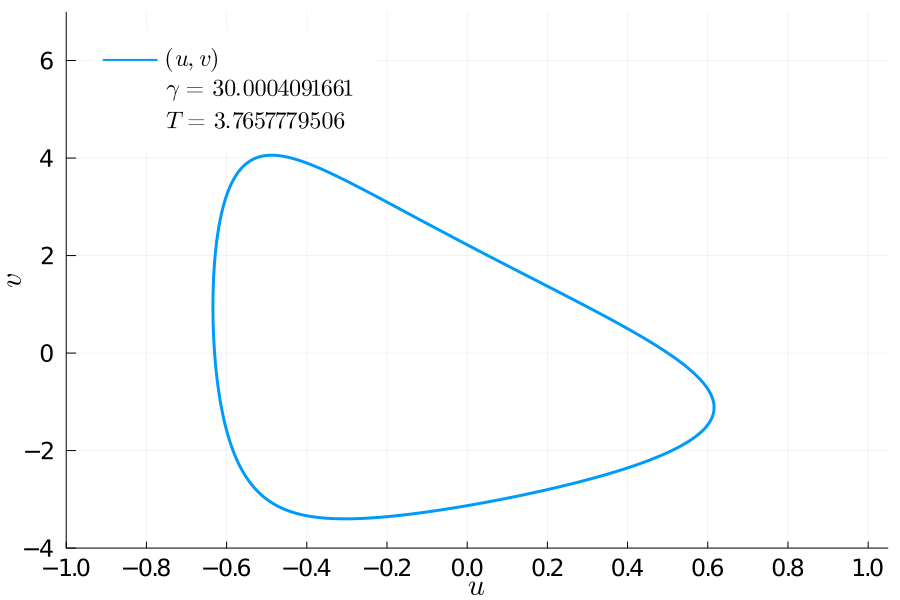}
 }
 \hspace{2mm}
  \subfloat[$s_{0} = 5$, principal branch.]
 {
    \includegraphics[width=0.31\textwidth]{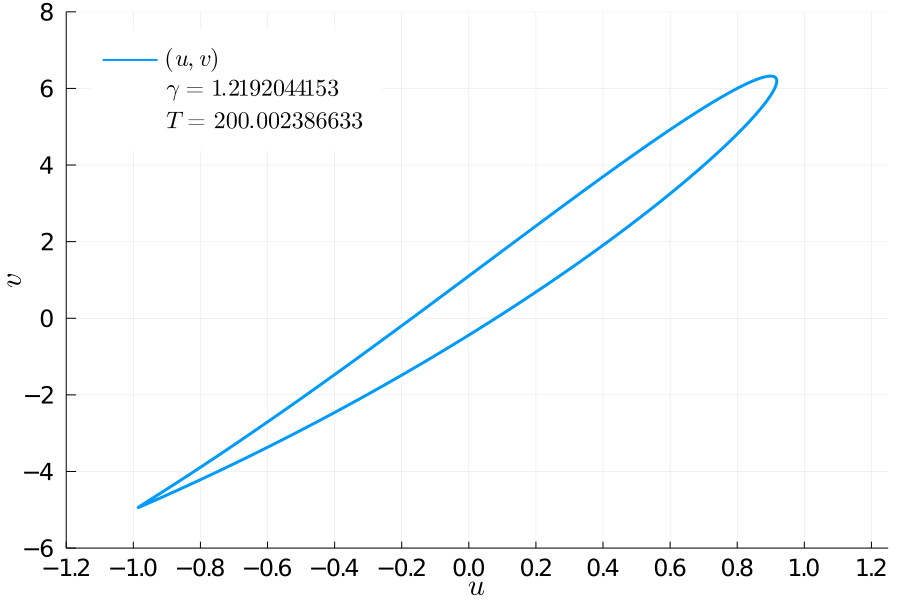}
 }
  \hspace{2mm}
  \subfloat[$s_{0} = 5$, second branch.]
 {
    \includegraphics[width=0.31\textwidth]{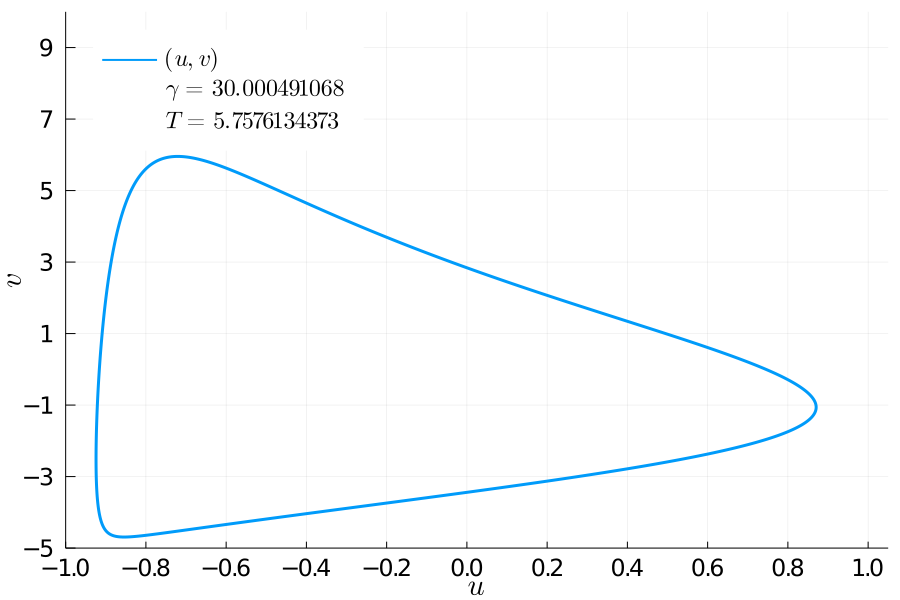}
 }
  \hspace{2mm}
  \subfloat[$s_{0} = 10$, principal branch.]
 {
    \includegraphics[width=0.31\textwidth]{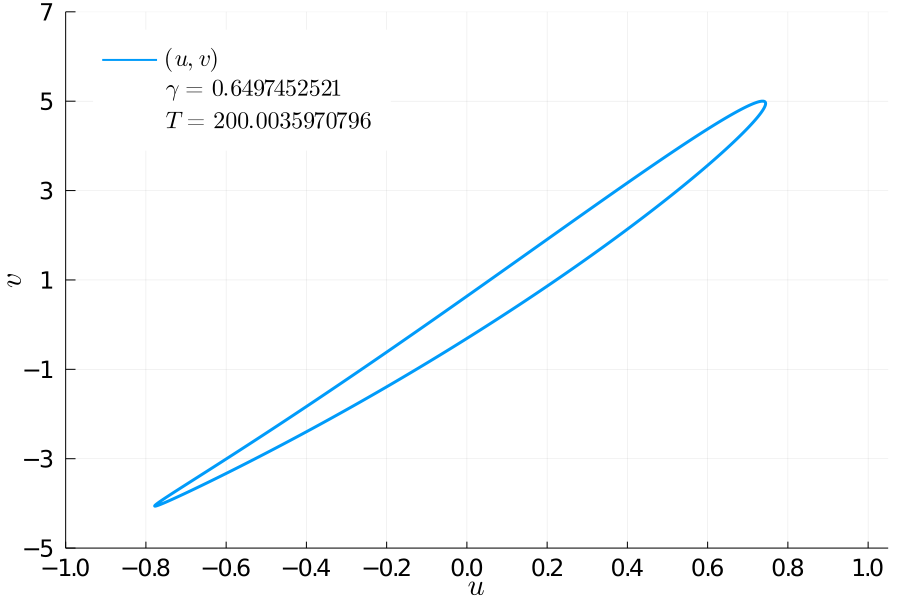}
 }
  \hspace{2mm}
  \subfloat[$s_{0} = 10$, second branch.]
 {
    \includegraphics[width=0.31\textwidth]{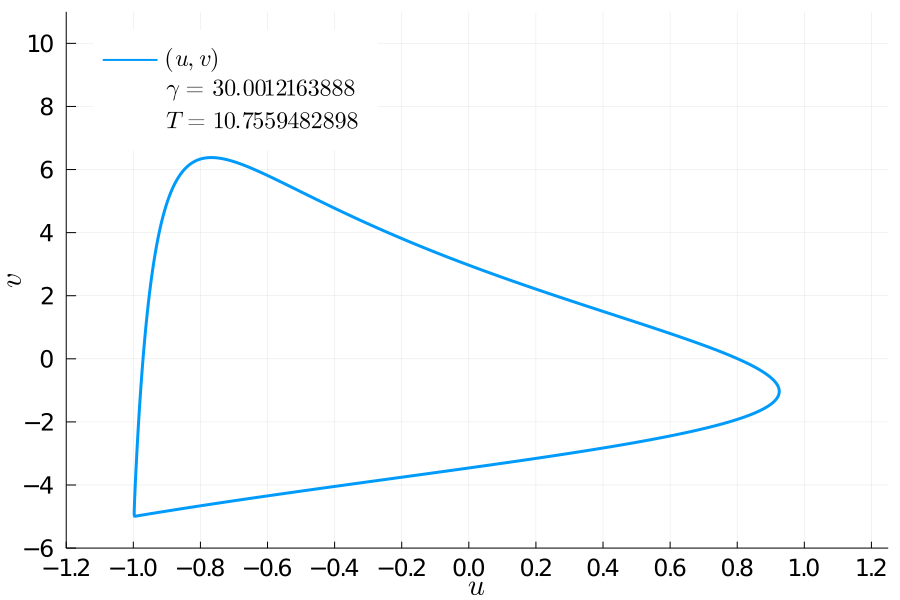}
 }
  \hspace{2mm}
  \subfloat[$s_{0} = 10$, third branch.]
 {
    \includegraphics[width=0.31\textwidth]{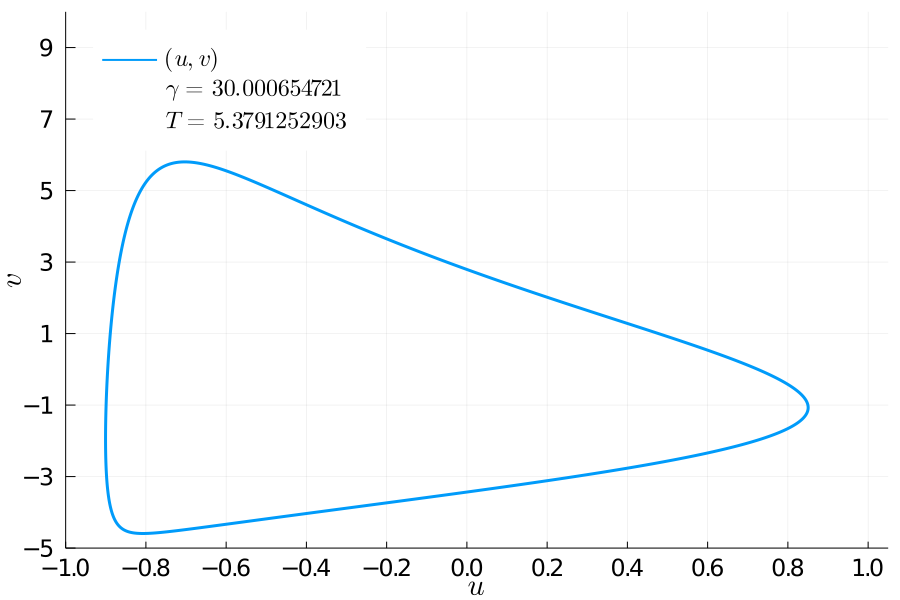}
 }
  \hspace{2mm}
  \subfloat[$s_{0} = 10$, fourth branch.]
 {
    \includegraphics[width=0.31\textwidth]{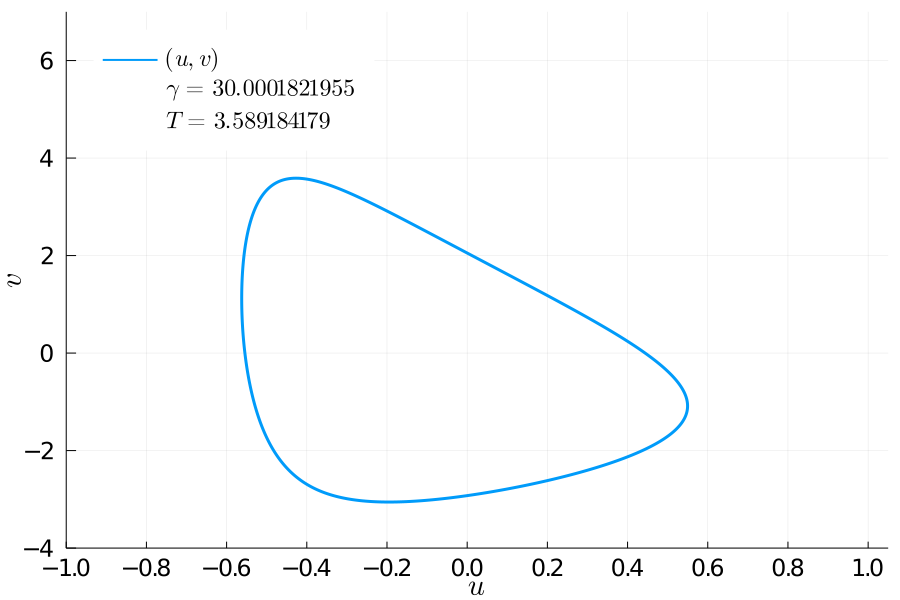}
 }
 \caption{Most right periodic solution of the principal branch of Hopf Bifurcation for Parameter Set 1 of the model (\ref{DDE_GM_around_0}), which lie to phase space $\mathcal{C}_{2,s0}$. A natural embedding of the solutions of our system into two dimensions are depicted in the $\left( u_{t}(0), v_{t}(0)\right)-$plane.}
 \label{f:most_right_periodic_solutions_Par_set_1}
\end{figure}
\section{Appendix}
We consider the Taylor expansion of the function $e (\varepsilon) = e^{-ins_{0}\omega (\varepsilon)}$, where $s_{0}\in\mathbb{R}$ and $n\in\mathbb{N}_{0}$. Using the expansion in \cite{haro_canadell_figueras_luque_mondelo} obtain $e(\varepsilon) = \sum\limits_{j=0}^{\infty}e_{j}\varepsilon^{j}$, where,

\begin{equation*}
e_{j}
=
\left\{
\begin{array}{l} 
e^{-ins_{0}\omega_{0}}, \text{ } j = 0, \\ \\
\dfrac{1}{j} \sum\limits_{l=0}^{j-1} (j-l) (-ins_{0} \omega_{j-l}) e_{l}, \text{ } j\geq 1.
\end{array}
\right.
\end{equation*}

\begin{proposition}
The coefficients $e_{j}$ only depend of $s_{0}$, $n$ and $\omega_{0},\dots,\omega_{j}$ for all $j \in \mathbb{N}_{0}$, i.e. $e_{j} = e_{j}(\omega_{0},\ldots,\omega_{j}; n, s_{0})$
\end{proposition}

\begin{proof}
We proceed by induction on $j$. The affirmation is true for $k\in\left\lbrace 0,1 \right\rbrace$ since $e_{0} = e^{-i n s_{0} \omega_{0}}$ and $e_{1} = -i n s_{0} \omega_{1} e^{-ins_{0}\omega_{0}}$.

Suppose that $e_{0},e_{1},\dots, e_{j-1}$ satisfies the property. Now, we show that $e_{j}$ satisfies the property. This is inmediate because using the definition and induction hypothesis,
	\begin{equation*}
		\begin{split}
			e_{j}
			&=
			\dfrac{1}{j} \sum\limits_{l=0}^{j-1} (j-l)(-in s_{0} \omega_{j-l}) e_{l}, \\
			&=
			\dfrac{1}{j} \sum\limits_{l=0}^{j-1} (j-l)(-in s_{0} \omega_{j-l})e_{l}(\omega_{0},\dots,\omega_{l};n,s_{0}).
		\end{split}
	\end{equation*}
\end{proof}
Thus, by Taylor theorem, $\dfrac{d^{j}}{d\varepsilon^{j}} \left(e^{-nis_{0}\omega(\varepsilon)}\right)\bigg\vert_{\varepsilon=0} = j! e_{j} (\omega_{0},\dots,\omega_{j};n,s_{0})$. In particular, for $k \geq 3$,
\begin{equation*}
e_{k-1} (\omega_{0},\dots,\omega_{k-1};n,s_{0})
=
-ins_{0}\omega_{k-1} e^{-ins_{0}\omega_{0}} + \tilde{e}_{k-2} (\omega_{0},\dots,\omega_{k-2};n,s_{0}),
\end{equation*}
where $\tilde{e}_{k-2} (\omega_{0},\dots,\omega_{k-2};n,s_{0}) = \dfrac{1}{k-1} \sum\limits_{l=1}^{k-2} (k-1-l)(-ins_{0}\omega_{k-1-l}) e_{l}(\omega_{0},\dots,\omega_{l};n,s_{0})$.

\section{Acknowledgment}
We would like to express our gratitude to the following organizations for their support: DGAPA-UNAM through project PAPIIT IN103423, IN104725, CONACYT for the PhD fellowship.
\bibliographystyle{alpha}
\bibliography{library1}
\end{document}